\documentclass[onefignum,onetabnum]{siamart171218}

\usepackage{amsfonts}
\usepackage{stmaryrd}
\SetSymbolFont{stmry}{bold}{U}{stmry}{m}{n}
\usepackage{bm}
\usepackage{upgreek}
\usepackage{accents}
\usepackage{color}      % colored text
\usepackage{enumerate}
\usepackage{graphicx}
\usepackage{epstopdf}
\usepackage{subfig}
\usepackage{algorithm}
\usepackage{algorithmic}
\usepackage{booktabs}
\usepackage{tabulary}
\newcolumntype{K}[1]{>{\centering\arraybackslash}p{#1}}
\usepackage[textsize=footnotesize,backgroundcolor=tan]{todonotes}
\ifpdf
  \DeclareGraphicsExtensions{.eps,.pdf,.png,.jpg}
\else
  \DeclareGraphicsExtensions{.eps}
\fi

\graphicspath{{../Figures/}}

\newsiamremark{remark}{Remark}
\newsiamremark{hypothesis}{Hypothesis}
\crefname{hypothesis}{Hypothesis}{Hypotheses}
\newsiamthm{claim}{Claim}

\headers{The HHJ Method With Curved Elements}{D. N. Arnold and S. W. Walker}

\title{The Hellan--Herrmann--Johnson Method \\ With Curved Elements\thanks{Submitted to the editors DATE.
\funding{This article is based upon work supported by the National Science Foundation under grants DMS-1719694 (Arnold) and DMS-155222 (Walker)
and by the Simons Foundations under grant 601937 (Arnold).}}}

\author{Douglas N. Arnold\thanks{Department of Mathematics, University of Minnesota, Minneapolis (\email{arnold@umn.edu}, \url{http://umn.edu/\~arnold/}).}
\and Shawn W. Walker\thanks{Department of Mathematics, Louisiana State University, Baton Rouge (\email{walker@lsu.edu}, \url{http://www.math.lsu.edu/\~walker/}).}}

\usepackage{amsopn}

\newcommand{\tr}{\operatorname{tr}}        %trace
\newcommand{\diam}{\operatorname{diam}}       %set diameter
\newcommand{\dd}{:} % matrix dot product
\newcommand{\tp}{^{T}} % \dag % matrix or vector transpose
\newcommand{\invtp}{^{-T}} % \dag % matrix or vector transpose
\newcommand{\CL}[1]{\overline{#1}} % closure of a set

\newcommand{\bg}{\bm{g}}

\newcommand{\bn}{\bm{n}}

\newcommand{\bt}{\bm{t}}
\newcommand{\bu}{\bm{u}}

\newcommand{\bB}{\bm{B}}

\newcommand{\bF}{\bm{F}}

\newcommand{\bJ}{\bm{J}}

\newcommand{\bR}{\bm{R}}

\newcommand{\bmeta}{\bm{\eta}}

\newcommand{\bxi}{\bm{\xi}}

\newcommand{\brho}{\bm{\rho}}

\newcommand{\bsigma}{\bm{\sigma}}

\newcommand{\btau}{\bm{\tau}}

\newcommand{\bvarphi}{\bm{\varphi}}

\newcommand{\bGamma}{\bm{\Gamma}}

\newcommand{\bPhi}{\bm{\Phi}}

\newcommand{\bPsi}{\bm{\Psi}}

\newcommand{\ve}{\mathbf{e}}

\newcommand{\vx}{\mathbf{x}}

\newcommand{\vI}{\mathbf{I}}

\newcommand{\cC}{\mathcal{C}}
\newcommand{\cD}{\mathcal{D}}

\newcommand{\cM}{\mathcal{M}}

\newcommand{\cV}{\mathcal{V}}
\newcommand{\cW}{\mathcal{W}}

\newcommand{\Dom}{\Omega} % the domain
\newcommand{\dDom}{\partial \Omega} % the boundary of domain
\newcommand{\BdyDom}{\Gamma}% the boundary of domain
\newcommand{\cornervtx}{\cV_{\BdyDom}} % set of corner (vertex) points
\newcommand{\smoothcurve}{\cC_{\BdyDom}} % set of smooth boundary curves
\newcommand{\Bclamped}{\BdyDom_{\mathrm{c}}} % boundary condition sets
\newcommand{\Bsimsupp}{\BdyDom_{\mathrm{s}}} % boundary condition sets
\newcommand{\darc}{d\mathrm{s}}
\newcommand{\darea}{d\mathrm{S}}

\newcommand{\ident}[1]{\mathrm{id}_{#1}}

\newcommand{\tanb}[1]{\ve_{#1}} % (covariant) tangent basis
\newcommand{\uu}{\bu} % the std vector u
\newcommand{\pd}[1]{\partial_{#1}} % partial derivative
\newcommand{\christ}[2]{\Gamma^{#1}_{#2}} % Christoffel 2nd kind

\newcommand{\kron}{\delta}
\newcommand{\met}[1]{g_{#1}}
\newcommand{\metmat}{\bg}
\newcommand{\invmet}[1]{g^{#1}}
\newcommand{\invmetmat}{\bg^{-1}}

\newcommand{\tvi}{t}
\newcommand{\btv}{\bt}

\newcommand{\cni}{n}
\newcommand{\bcn}{\bn}

\newcommand{\grad}[1]{\nabla_{#1}}
\newcommand{\hess}[1]{\nabla^2_{#1}}

\newcommand{\Div}[1]{\mathrm{div}_{#1}\,}         %div
\newcommand{\DivDiv}[1]{\mathrm{div} \, \mathrm{div}_{#1}\,}  %div div

\newcommand{\symmat}{\mathbb{S}}   % space of symmetric matrices
\newcommand{\R}{\mathbb{R}}   % real numbers
\newcommand{\zerobdy}[1]{\mathring{#1}}   % make space have zero boundary data
\newcommand{\LAG}{\W}   % function space
\newcommand{\HHJ}{\V}   % function space
\newcommand{\V}{V}   % function space
\newcommand{\W}{W}   % function space
\newcommand{\SW}{\cW}   % function space
\newcommand{\SV}{\cV}   % function space
\newcommand{\Mnn}{\cM_{\mathrm{nn}} } % "smooth" normal-normal space
\newcommand{\Hdiv}[1]{H(\mathrm{div}, #1)} % std H(div) space
\newcommand{\inner}[3]{\left( #1, #2 \right)_{#3}}  % L^2 inner product
\newcommand{\duality}[3]{\left\langle #1, #2 \right\rangle_{#3}}  % L^2 inner product

\newcommand{\afman}[3]{a^{#1} \left( #2, #3 \right)}  % tensor form
\newcommand{\bfman}[3]{b_{h}^{#1} \left( #2, #3 \right)}  % div div form
\newcommand{\afcont}{A_{0}}
\newcommand{\bfcont}{B_{0}}
\newcommand{\afcoer}{\alpha_{0}}
\newcommand{\bfcoer}{\beta_{0}}
\newcommand{\ringbfman}[3]{\mathring{b}_{h}^{#1} \left( #2, #3 \right)}  % no bdy term
\newcommand{\jump}[1]{\left\llbracket #1 \right\rrbracket}
\newcommand{\trinorm}[1]{{\left\vert\kern-0.25ex\left\vert\kern-0.25ex\left\vert #1 \right\vert\kern-0.25ex\right\vert\kern-0.25ex\right\vert}}

\newcommand{\signn}{\sigma^{\mathrm{nn}}}
\newcommand{\hatsignn}{\hat{\sigma}^{\mathrm{nn}}}

\newcommand{\varphinn}{\varphi^{\mathrm{nn}}}

\newcommand{\bvarphiT}{\bvarphi'}
\newcommand{\hatvarphinn}{\hat{\varphi}^{\mathrm{nn}}}

\newcommand{\eye}{\vI}
\newcommand{\meshpoincare}{C_{\mathrm{P}}}

\newcommand{\bendmod}{D}
\newcommand{\youngmod}{E}
\newcommand{\poisson}{\nu}
\newcommand{\shearmod}{\bar{\mu}}

\newcommand{\belasten}{\mathbf{C}}
\newcommand{\invelasten}[4]{K_{#1 #2 #3 #4}}
\newcommand{\binvelasten}{\mathbf{K}}

\newcommand{\hatinvelasten}[4]{\widehat{K}_{#1 #2 #3 #4}}

\newcommand{\TkDom}[2]{\mathcal{T}_{#1}^{#2}}
\newcommand{\TkBdyDom}[2]{\mathcal{T}_{\partial,#1}^{#2}} % triangles with one edge on the boundary

\newcommand{\EkDom}[2]{\mathcal{E}_{#1}^{#2}} % set of edges
\newcommand{\EkIntDom}[2]{\mathcal{E}_{0,#1}^{#2}} % internal edges
\newcommand{\EkBdyDom}[2]{\mathcal{E}_{\partial,#1}^{#2}} % bdy edges
\newcommand{\strip}{\Dom_{S}}
\newcommand{\Pk}{\mathcal{P}} % polynomial space symbol

\newcommand{\Ilag}{\mathcal{I}} % HHJ lagrange interp oper
\newcommand{\LtwoProj}{\mathrm{P}_{h}} % projection operator
\newcommand{\CProj}{\mathrm{P}_{0}} % projection operator onto piecewise constants
\newcommand{\IHHJ}{\Pi} % HHJ tensor interp oper
\newcommand{\MapT}[1]{\bF_{#1}}
\newcommand{\Jac}[1]{\bB_{#1}}

\definecolor{tan}{rgb}{1,0.8,0.7}

\ifpdf
\hypersetup{
  pdftitle={The Hellan--Herrmann--Johnson Method With Curved Elements},
  pdfauthor={D. N. Arnold and S. W. Walker}
}
\fi

\begin{document}

\maketitle

% REQUIRED
\begin{abstract}
	We study the finite element approximation of the Kirchhoff plate equation on domains with curved boundaries using the Hellan--Herrmann--Johnson (HHJ) method.  We prove optimal convergence on domains with piecewise $C^{k+1}$ boundary for $k \geq 1$ when using a parametric (curved) HHJ space.  
	Computational results are given that demonstrate optimal convergence and how convergence degrades when curved triangles of insufficient polynomial degree are used.  
	Moreover, we show that the lowest order HHJ method on a polygonal approximation of the disk does not succumb to the classic Babu\v{s}ka paradox, highlighting the geometrically non-conforming aspect of the HHJ method.
\end{abstract}

% REQUIRED
\begin{keywords}
  Kirchhoff plate, simply-supported, parametric finite elements, mesh-dependent norms, geometric consistency error, Babu\v{s}ka paradox
\end{keywords}

% REQUIRED
\begin{AMS}
  65N30, 35J40, 35Q72
\end{AMS}

%------------------------------------------------------------------------
\section{Introduction}\label{sec:intro}
%------------------------------------------------------------------------

The fourth order Kirchhoff plate bending problem presents notorious difficulties for finite element discretization. Among the many approaches that have been proposed, the Hellan--Herrmann--Johnson (HHJ) mixed method is one of the most successful.   In simple situations (polygonal domains and smooth solutions) it provides stable discretization of arbitrary order, and has been analyzed by many authors
\cite{Brezzi_RAIRO1975,Brezzi_RIA1976,Brezzi_Cal1980,Babuska_MC1980,Arnold_M2AN1985,Comodi_MC1989,Blum_CM1990,Stenberg_M2AN1991,Krendl_ETNA2016,Rafetseder_SJNA2018}.  However, in realistic applications the plate domain may well have a curved boundary, and additional errors arise from geometric approximation of the domain.  In this paper, we analyze the effect of this geometric approximation, and show that, if handled correctly, the full discretization converges at the same optimal rate as is achieved for polygonal plates.

The well known Babu{\v s}ka paradox demonstrates that there may be difficulties with low degree approximation of the geometry.  Specifically, the paradox considers the effect of approximating the geometry only, without further numerical error.  It considers a uniformly loaded simply-supported circular plate and approximates the solution by the \emph{exact solution} of the same problem on an inscribed regular polygon.  As the number of sides of the polygon increases, the solution does \emph{not} converge to the solution on the disk.  The errors arising from linear approximation of the geometry leads to nonconvergence.  However, as we shall show below, if the problem on the polygon is solved using the lowest order HHJ method, then convergence is restored.  We further show that if higher order approximation of the geometry is combined with HHJ discretization of higher degree, the resulting method achieves any desired order.

While, to the best of our knowledge, the effect of domain approximation has not been studied before for the HHJ discretization of plates, its effect on the solution of second order problems by standard finite elements is classical.  See for instance, \cite{Strang_PDEproc1971,LRScott_PhD1973,Ciarlet_CMAME1972,Lenoir_SJNA1986}.  For simplicity, consider the Poisson equation $-\Delta u = f$ on $\Dom$ with the Dirichlet boundary condition $u=0$ on $\partial \Dom$.  Suppose we approximate the domain using parametric curved elements of some degree $m$ to obtain an approximate domain $\Dom^m$ consisting of elements with maximum diameter $h$.  To compare approximate solutions obtained on the approximation domain to the true solution on the true domain, we require a diffeomorphic mapping $\bPsi : \Dom^m \to \Dom$.  The error analysis
then depends on the identity \cite[Sec. 6]{Lenoir_SJNA1986}
\begin{equation*}
	\int_{\Dom} \grad{} u \cdot \grad{} v = \int_{\Dom^m} \grad{} \tilde{u} \cdot \grad{} \tilde{v} + \int_{\Dom^m} \grad{} \tilde{u} \cdot \left[ \det (\bJ) \bJ^{-1} \bJ\invtp - \eye \right] \grad{} \tilde{v},
\end{equation*}
where $\tilde{u} = u \circ \bPsi$ (and similarly for $v$) and $\bJ = \grad{} \bPsi$ is the Jacobian matrix.  The
mapping $\bPsi$ is defined so that $\| \bJ - \eye \|_{L^{\infty}} = O(h^{m})$ where $m$ is the degree of the polynomials used for the domain approximation.  This leads to an $O(h^m)$ bound on the geometric consistency error term, the second integral on the right-hand side of the above identity.  Choosing $m$ to equal or exceed the degree of the finite elements used to approximate the solution (isoparametric or superparametric approximation) then ensures that full approximation order is maintained with curved elements.

Convergence for fourth order problems is less well established.  In \cite{Monk_SJNA1987}, the biharmonic problem is split into two second order equations with curved isoparametric elements and slightly modified boundary conditions.  For plate problems, analysis of $C^1$ domain approximations have been considered (see \cite{Scott_Report1976,Chernuka_IJNME1972,Mansfield_SJNA1978,Tiihonen_MC2001}).

The purpose of this paper is to give a rigorous estimate of the error between the continuous solution on the true domain and the discrete solution on the approximate domain.
\emph{The main difficulty} in this is dealing with higher derivatives of the nonlinear map that appear in the analysis (for instance, see \cite[pg. 78]{Boffi_book2013} and \cite[Thm. 4.4.3]{Ciarlet_Book2002}).  For example, when mapping the Hessian, we have
$\hess{} v = \bJ^{-1} \left[ \hess{} \tilde{v} - \pd{\gamma} \tilde{v} \bGamma^{\gamma} \right] \bJ\invtp$, where $\bGamma^{\gamma}$ is a $2\times 2$ matrix whose entries are the Christoffel symbols $\christ{\gamma}{\alpha \beta}$ of the second kind for the induced metric.  These depend on \emph{second} derivatives of $\bPsi$, and, consequently, $\| \bGamma^{\gamma} \|_{L^{\infty}} = O(h^{m-1})$, so a naive handling of this term would yield sub-optimal results or no convergence at all for $m=1$.  Another related issue is the handling of jump terms (appearing in some mesh dependent norms) when affected by the nonlinear map.

The crucial tools needed to overcome these difficulties is the use of a Fortin-like operator \cref{eqn:b_h_form_interp_oper_Fortin} together with a particular optimal map \cref{ass:HHJ_Lag_interp_optimal_map} that is different from the curved element map given in \cite{LRScott_PhD1973,Lenoir_SJNA1986}.  The results we present here should be of relevance to simulating plate problems on smooth and piecewise smooth domains.

We close the introduction with a brief outline of the remainder of the paper.
\Cref{sec:plate_eqn} reviews the Kirchhoff plate problem and the mesh-dependent weak formulation behind the HHJ method.  \Cref{sec:curved_FEM} provides a quick review of curved finite elements and \cref{sec:manifold_HHJ} shows how to extend the classic HHJ method to curved elements.  \Cref{sec:error_analysis} provides the error analysis, which follows the framework of \cite{Babuska_MC1980} and \cite{Blum_CM1990}, where we use a formulation of the Kirchhoff plate problem based on mesh dependent spaces, and analyze it with mesh-dependent norms. \Cref{sec:numerical_results} gives numerical results and we conclude in \cref{sec:conclusions} with some remarks.  We also collect several basic or technical results in the supplementary materials.

%------------------------------------------------------------------------
\section{Preliminaries and statements of results}\label{sec:plate_eqn}
%------------------------------------------------------------------------
We begin by recalling the Kirchhoff plate problem and the HHJ discretization, and establishing our notations.
The domain of the plate, i.e., its undeformed mid-surface, is denoted by $\Dom \subset \R^2$ and its boundary by $\BdyDom := \dDom$.
Denoting the vertical displacement by $w$ and the bending moment tensor by $\bsigma$, the plate equations \cite[pg. 44--51]{Landau_Book1970} are
\begin{equation}\label{eqn:plate_eqs}
\bsigma = \belasten \,\hess{} w,\quad
\DivDiv{}\bsigma = f \text{ in $\Dom$},\quad
\belasten\btau:= \bendmod \left[ (1 - \poisson) \btau + \poisson \tr(\btau) \eye \right],
\end{equation}
where $\hess{} w$ denotes the Hessian of $w$, the iterated divergence $\DivDiv{}$ takes a matrix field to a scalar function, $f$ denotes the load function, and $\belasten$ is the constitutive tensor, 
with the bending modulus $\bendmod$ given by $\youngmod t^3/12 (1 - \poisson^2)$
in terms of Young's modulus $\youngmod$, the Poisson ratio $\poisson$, and the plate thickness $t$.  We assume that $\poisson \in (-1,1)$, so $\belasten$ is a symmetric positive-definite operator on the space $\symmat$ of symmetric $2\times 2$ tensors. For a standard material, $0 \leq \poisson < 1/2$.

The differential equations \cref{eqn:plate_eqs} are supplemented by boundary conditions on $\dDom$, such as
\begin{equation*}
w= \pd{} w / \pd{} \cni =0, \text{ for a clamped plate, or } ~ w= \signn = 0, \text{ for a simply-supported one}.
\end{equation*}
Here $\sigma^{\mathrm{nn}}=\bcn\tp\bsigma\bcn$ denotes the normal-normal component of $\bsigma$.

The Kirchhoff plate problem can be formulated weakly.  Taking $\SW=\zerobdy{H}^{2} (\Dom)$ or
$H^{2} (\Dom) \cap  \zerobdy{H}^{1} (\Dom)$ for clamped and simply-supported boundary conditions, respectively, $w\in\SW$ is uniquely determined by the weak equations
\begin{equation}\label{eqn:weak_form}
\inner{\belasten \hess{} w}{\hess{} v}{} = \duality{f}{v}{}, \text{ for all $v \in \SW$,}
\end{equation}
for any $f$ in $L^2(\Dom)$, or, more generally, in $\SW^*$.  Note that we use standard notations $H^m(\Dom)$ and $\zerobdy H^m(\Dom)$ for Sobolev spaces, with the latter subject to vanishing traces.

Next, we recall the HHJ method, first in the case of a polygonal domain (or a polygonal approximation to the true domain), and then for higher order polynomial approximations to the domain, which is the main subject here.  Let $\TkDom{h}{}$ be a triangulation of the polygonal domain $\Dom$
and let the degree $r\ge 0$ be fixed.  The transverse displacement $w$ will be approximated in the usual Lagrange finite element space
\begin{equation*}
\LAG_{h} = \{ v \in \zerobdy{H}^{1} (\Dom) \mid v |_{T} \in \Pk_{r+1}(T) ~ \forall T \in \TkDom{h}{} \},
\end{equation*}
while the bending moment tensor $\bsigma$ will be sought in the HHJ space
\begin{equation}\label{eqn:discrete_tensor_space}
\HHJ_{h} = \{ \bvarphi \in  L^2 (\Dom;\symmat) \mid \bvarphi |_{T} \in\Pk_r(T;\symmat)~ \forall T \in \TkDom{h}{}, \text{ $\bvarphi$ normal-normal continuous}\}.
\end{equation}
The normal-normal continuity condition means that, if two triangles $T_1$ and $T_2$ share a common edge $E$, then $\bcn\tp(\bvarphi|_{T_1})\bcn = \bcn\tp(\bvarphi|_{T_2})\bcn$ on $E$. 
For simply-supported boundary conditions, the space $\HHJ_{h}$ also incorporates the vanishing of $\varphinn$ on boundary edges.

Assume that $\bsigma$ belongs to $H^1(\Dom;\symmat)$ and $f$ belongs to $L^2(\Dom)$ (this is for simplicity; it can be weakened).  Multiplying the second equation in \cref{eqn:plate_eqs} by a test function $v \in \LAG_h$ and integrating over a triangle $T$, we obtain
\begin{align*}
(f, v)_T &= (\DivDiv{}\bsigma,v)_T = -(\Div{}\bsigma,\grad{}v)_T=(\bsigma,\hess{} v)_T - \langle \bsigma\bcn, \grad{}v\rangle_{\partial T} 
\\ &= (\bsigma,\hess{} v)_T - \langle \bcn^T\bsigma\bcn, \pd{} v/\pd{} \cni\rangle_{\partial T} - \langle \bt^T\bsigma\bcn, \pd{} v/\pd{} \tvi\rangle_{\partial T}.
\end{align*}
Next, we sum this equation over all the triangles $T$.  The penultimate term gives
\begin{equation*}
\sum_T\langle \bcn^T\bsigma\bcn, \pd{} v/\pd{} \cni\rangle_{\partial T} 
= \sum_{E \in \EkDom{h}{}} \duality{\signn}{\jump{\pd{} v/\pd{} \cni}}{E}.
\end{equation*}
Here $\jump{\eta}$ denotes the jump in a quantity $\eta$ across a mesh edge $E$, so if the edge $E$ is shared by two triangles $T_1$ and $T_2$
with outward normals $\bcn_1$ and $\bcn_2$, then $\jump{\pd{} v/\pd{} \cni} = \bcn_1\cdot\grad{} v|_{T_1} + \bcn_2\cdot\grad{}v|_{T_2}$ on $E$.
For $E$ a boundary edge, we set $\jump{\eta}=\eta|_E$. For the final term above, we obtain $\sum_T\langle \bt^T\bsigma\bcn, \pd{} v/\pd{} \tvi\rangle_{\partial T}=0$, since $\pd{} v/\pd{} \tvi$ is continuous across interior edges and the normal vector switches sign; moreover, $\pd{} v/\pd{} \tvi$ vanishes on boundary edges. Thus, if we define the bilinear form
\begin{equation*}
 \bfman{}{\bvarphi}{v} = -\sum_{T \in \TkDom{h}{}} \inner{\bvarphi}{\hess{} v}{T} + \sum_{E \in \EkDom{h}{}} \duality{\varphinn}{\jump{\pd{} v/\pd{} \cni}}{E},\quad
 \bvarphi\in \HHJ_{h},\ v\in \LAG_{h},
\end{equation*}
we have $\bfman{}{\bsigma}{v}  = -\duality{f}{v}{}$ for all $v \in \LAG_{h}$.
We define as well a second bilinear form 
\begin{equation*}
\afman{}{\btau}{\bvarphi} = \inner{\binvelasten \btau}{\bvarphi}{}, ~~ \btau,\bvarphi\in \HHJ_{h}, ~~ \text{where }~ \binvelasten \btau := \frac{1}{\bendmod} \left[ \frac{1}{1 - \poisson} \btau - \frac{\poisson}{1 - \poisson^2} \tr (\btau) \eye \right],
\end{equation*}
with $\binvelasten$ the inverse of $\belasten$.  Using the first equation in \cref{eqn:plate_eqs} and the continuity of $\pd{} w/\pd{} \cni$, we have
$\afman{}{\bsigma}{\btau} + \bfman{}{\btau}{w} = 0$ for any $\btau\in\HHJ_{h}$.  This leads us to
the HHJ mixed method, which defines $\bsigma_h\in \HHJ_{h}$, $w_h\in\LAG_{h}$ by
\begin{equation}\label{eqn:HHJ_method}
\begin{gathered}
\afman{}{\bsigma_h}{\btau} + \bfman{}{\btau}{w_h} = 0, \quad \forall \btau \in \HHJ_{h}, \\
\bfman{}{\bsigma_h}{v} = -\duality{f}{v}{}, ~ \forall v \in \LAG_{h}.
\end{gathered}
\end{equation}
This method has been analyzed by numerous authors with different techniques.  The present analysis owes the
most to \cite{Babuska_MC1980} and \cite{Blum_CM1990}.  In particular, optimal
$O(h^{r+1})$ convergence for $\bsigma$ in $L^2$ and $w$ in $H^1$ has been established for smooth solutions.

If the plate domain $\Dom$ is not polygonal, a simple possibility is to construct a polygonal approximate domain.  For this, we let $\TkDom{h}{1}$ denote a triangulation consisting of straight-edged triangles with interior vertices belonging to $\Dom$ and boundary vertices belonging to $\dDom$.  The approximate domain $\Dom^1$ is the region triangulated by $\TkDom{h}{1}$.  We assume further that no element of $\TkDom{h}{1}$ has more than one edge on the boundary of $\Dom^1$, and call those that have such an edge boundary triangles.

\begin{figure}[ht]
\begin{center}
\includegraphics[width=2.2in]{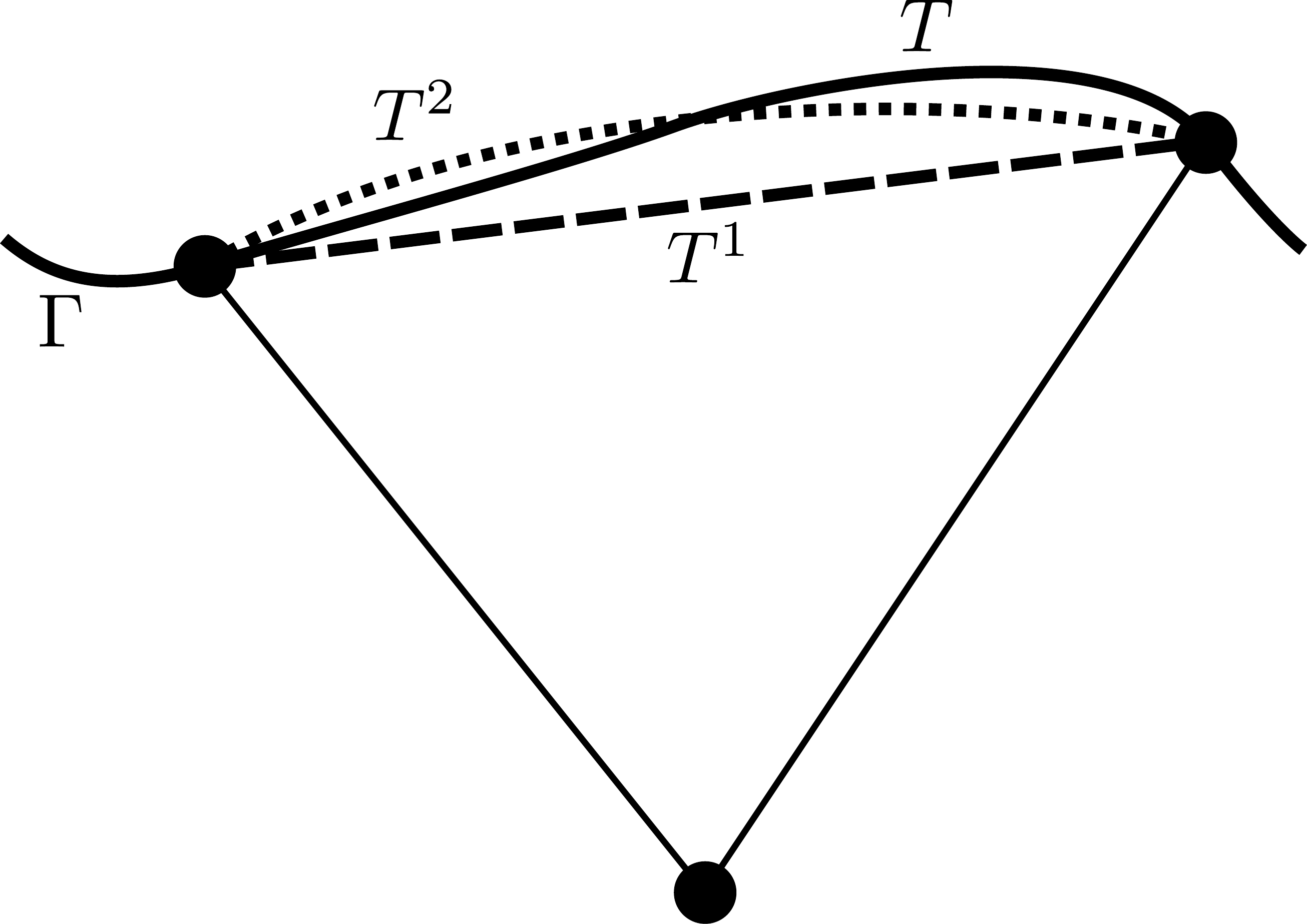}
\caption{A straight-edged triangle $T^1$, degree $2$ curvilinear triangle $T^2$, and curvilinear triangle $T$ exactly conforming to the boundary $\Gamma$. All three share the two straight edges, with the corresponding boundary edges shown as dashed, dotted, and solid, respectively.}
\label{fig:triangles}
\end{center}
\end{figure}

As we shall see, in the case of the lowest order HHJ elements, $r=0$, such a polygonal approximation of the geometry does not degrade the rate of convergence of the numerical scheme (the Babu\v ska paradox notwithstanding).  For higher order elements, however, we need to make a better approximation of the geometry in order to obtain the approximate rate, just as is true when solving the Poisson problem with standard Lagrange finite elements \cite{Lenoir_SJNA1986}.  We now briefly describe the procedure (see \cref{fig:triangles}), the full specification and analysis of which will occupy the remainder of the paper.  Let $m\ge 1$ denote the integer degree of approximation of the geometry (so $m=1$ corresponds to the polygonal approximation).  To each triangle $T^1\in\TkDom{h}{1}$ we associate a curvilinear triangle $T^m$ and  a diffeomorphism $\MapT{T}^m:T^1\to T^m$ which is a polynomial map of degree $m$.  In the case where $T^1$ is a boundary triangle, we require that $\MapT{T}^m$ restricts to the identity on the two non-boundary edges of $T^1$, and in case $T^1$ is not a boundary triangle, we simply take $T^m=T^1$ and $\MapT{T}^m$ to be the identity.  We require that the set of all such curvilinear triangles forms a triangulation $\TkDom{h}{m}$ of the domain $\Dom^m =\bigcup_{T^1\in\TkDom{h}{}} \MapT{T}^m(T^1)$, which is a polynomial approximation of the true domain of degree $m$.  Note that, the map $\MapT{}^m:\Dom^1\to \Dom^m$ given by $\MapT{}^m|_{T^1} =\MapT{T}^m$, for all $T^1\in\TkDom{h}{1}$, is a diffeomorphism of the polygonal approximate domain onto the approximate domain of degree $m$.

Using the mapping $\MapT{}^m$ we may transform the finite element spaces $\LAG_{h}$ and $\HHJ_{h}$ from the polygonal approximate domain to the degree $m$ approximation $\Dom^m$.  For $\LAG_{h}$, the transformation is a simple composition, but for the tensor space $\HHJ_{h}$ we must use the \emph{matrix Piola transform}, which preserves normal-normal continuity.  In this way, we obtained a mixed discretization of the plate problem based on elements of degree $r+1$ and $r$ for $w$ and $\bsigma$, respectively, and geometric approximation of degree $m$. The integers $r\ge0$ and $m\ge1$ can be taken arbitrarily, but we show that to obtain the same optimal rates of convergence on a curved domain, as occurs for smooth solutions on a polygonal domain, it is sufficient to take $m\ge r+1$, e.g., polygonal approximation ($m=1$) is sufficient for the lowest order HHJ elements ($r=0$), but
approximation must be at least quadratic to obtain optimality when $r=1$, and so forth.  Numerical experiments are included to show the necessity of this restriction.

%------------------------------------------------------------------------
\subsection{Boundary Assumptions}\label{sec:bdy_assume}
%------------------------------------------------------------------------

\begin{figure}[ht]
\begin{center}

\includegraphics[width=3.0in]{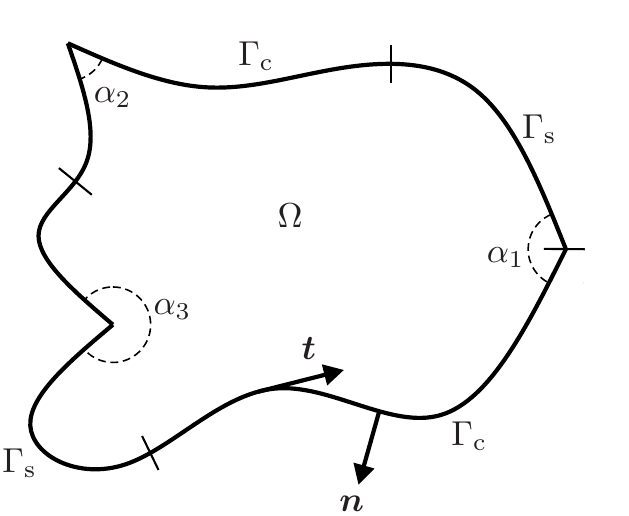}
\caption{Illustration of plate domain $\Dom$.  The boundary $\BdyDom$ decomposes as $\BdyDom = \CL{\Bclamped} \cup \CL{\Bsimsupp}$ and has a finite number of corners with interior angles $\alpha_{i}$; the corners may (or may not) lie at the intersection of two boundary components.  The outer unit normal vector is $\bcn$ and oriented unit tangent vector is $\btv$.
}
\label{fig:plate_domain}
\end{center}
\end{figure}

We shall allow for mixed boundary conditions, clamped on part of the domain, and simply-supported on the rest.  To this end, we
assume that $\BdyDom$ is piecewise smooth with a \emph{finite} number of corners, where the interior angle $\alpha_i$ of the $i$th corner satisfies $\alpha_i \in (0, 2\pi]$ (see \cref{fig:plate_domain}).  In particular, $\BdyDom$ is globally continuous and can be parameterized by a piecewise $C^{k+1}$ curve for some $k \geq 1$, i.e., $\BdyDom = \bigcup_{p \in \cornervtx} p \; \cup \bigcup_{\zeta \in \smoothcurve} \zeta$, 
where $\cornervtx$ is the set of corner vertices and $\smoothcurve$ is the set of (open) $C^{k+1}$ curves that make up $\BdyDom$.
Moreover, we assume $\BdyDom = \CL{\Bclamped} \cup \CL{\Bsimsupp}$ partitions into two mutually disjoint one dimensional components $\Bclamped$ (clamped) and $\Bsimsupp$ (simply supported).  Each open curve $\zeta \in \smoothcurve$ belongs to only one of the sets $\Bclamped$ or $\Bsimsupp$ and each curve is maximal such that two distinct curves contained in the same component do not meet at an angle of $\pi$.  At the expense of small additional technical and notational complications, we could allow a partition of the boundary into three sets rather than two, imposing free boundary conditions on the third portion.

With the above partition of $\BdyDom$, we have the following set of boundary conditions:
\begin{equation}\label{eqn:Kirchhoff_plate_BCs_general}
\begin{split}
	w &= \pd{} w / \pd{} \cni = 0, \text{ on } \Bclamped, \quad w = \bcn\tp \bsigma \bcn = 0, \text{ on } \Bsimsupp.
\end{split}
\end{equation}
Extending the definition of the energy space $\SW$ to account for these mixed boundary conditions,
\begin{equation}\label{eqn:displacement_space_true}
\begin{split}
\SW (\Dom) := \{ v \in H^2(\Dom) \mid v = 0, \text{ on } \Bclamped \cup \Bsimsupp, ~ \pd{\bcn} v = 0, \text{ on } \Bclamped \},
\end{split}
\end{equation}
we have that $\inner{\belasten \hess{} v}{\hess{} v}{} \geq a_0 \| v \|_{H^2(\Dom)}$ for all $v\in\SW$ and some constant $a_0 > 0$.  Consequently, there exists a unique $w\in\SW$ satisfying the the plate equations in the weak formulation \cref{eqn:weak_form}.

%------------------------------------------------------------------------
\subsection{Continuous Mesh-dependent Formulation}\label{sec:contin_non-conform_abstract}
%------------------------------------------------------------------------

The main difficulty in solving \cref{eqn:weak_form} numerically is that $\SW \subset H^2(\Dom)$ and so $C^1$ elements are required for a conforming discretization. 
We adopt the approach in \cite{Brezzi_RIA1976,Babuska_MC1980,Arnold_M2AN1985,Blum_CM1990} and use a mesh-dependent version of $H^2(\Dom)$.  We start by partitioning the domain $\Dom$ with a mesh $\TkDom{h}{} = \{ T \}$ of triangles such that $\Dom = \bigcup_{T \in \TkDom{h}{}} T$, where $h_T := \diam (T)$ and $h := \max_{T} h_T$, and assume throughout that the mesh is quasi-uniform and shape regular.  We further assume the corners of the domain are captured by vertices of the mesh.

Next, we have the \emph{skeleton} of the mesh, i.e., the set of mesh edges $\EkDom{h}{} := \partial \TkDom{h}{}$.  Let $\EkBdyDom{h}{} \subset \EkDom{h}{}$ denote the subset of edges that are contained in the boundary $\BdyDom$ and respect the boundary condition partition of $\BdyDom$.  The internal edges are given by $\EkIntDom{h}{} := \EkDom{h}{} \setminus \EkBdyDom{h}{}$. Note that elements in $\TkDom{h}{}$, $\EkDom{h}{}$ may be curved.  For now, we assume $\BdyDom$ is piecewise smooth (at least $C^2$) with a \emph{finite} number of corners to which the mesh conforms (see \cref{sec:bdy_assume} for more detailed assumptions).

The spaces in the following sections are infinite dimensional, but defined in a ``broken'' way with respect to the partition.  Thus, we adopt standard dG notation for writing inner products and norms over the partition, e.g.,
\begin{equation}\label{eqn:defn_broken_integrals}
\begin{split}
\inner{f}{g}{\TkDom{h}{}} &:= \sum_{T \in \TkDom{h}{}} \inner{f}{g}{T}, \quad \inner{f}{g}{\EkDom{h}{}} := \sum_{E \in \EkDom{h}{}} \inner{f}{g}{E}, \\
\| f \|_{L^p(\TkDom{h}{})}^{p} &:= \sum_{T \in \TkDom{h}{}} \| f \|_{L^p(T)}^{p}, \quad \| f \|_{L^p(\EkDom{h}{})}^{p} := \sum_{E \in \EkDom{h}{}} \| f \|_{L^p(E)}^{p}.
\end{split}
\end{equation}

We shall make repeated use of the following scaling/trace estimate \cite[Thm 3.10]{Agmon_book1965}:
\begin{equation}\label{eqn:scaling_trace_estimate}
\begin{split}
	\| v \|_{L^2(\partial T)}^2 &\leq C \left( h^{-1} \| v \|_{L^2(T)}^2 + h \| \grad{} v \|_{L^2(T)}^2 \right), \text{ for all } v \in H^1(T).
\end{split}
\end{equation}

%------------------------------------------------------------------------
\subsubsection{Skeleton Spaces}\label{sec:skeleton_spaces}
%------------------------------------------------------------------------

We follow \cite{Babuska_MC1980} in defining infinite dimensional, but mesh dependent spaces and norms.  A mesh-dependent version of $H^2(\Dom)$ is given by
\begin{equation}\label{eqn:scalar_H2_h}
H^2_{h} (\Dom) := \{ v \in H^1(\Dom) \mid v |_{T} \in H^2(T), \text{ for } T \in \TkDom{h}{} \},
\end{equation}
with the following semi-norm
\begin{equation}\label{eqn:scalar_H2_h_norm}
\begin{split}
\| v \|_{2,h}^2 &:= \| \hess{} v \|_{L^2(\TkDom{h}{})}^2 + h^{-1}  \left\| \jump{\bcn \cdot \grad{} v} \right\|_{L^2(\EkIntDom{h}{})}^2 + h^{-1} \left\| \jump{\bcn \cdot \grad{} v} \right\|_{L^2(\Bclamped)}^2,
\end{split}
\end{equation}
where $\jump{\eta}$ is the jump in quantity $\eta$ across mesh edge $E$, and $\bcn$ is the unit normal on $E \in \EkDom{h}{}$; on a boundary edge, $\jump{\eta} \equiv \eta$.
Next, for any $\bvarphi \in H^{1}(\Dom;\symmat)$ define
\begin{equation}\label{eqn:matrix_H0_h_norm}
\| \bvarphi \|_{0,h}^2 := \| \bvarphi \|_{L^2(\Dom)}^2 + h \left\| \bcn\tp \bvarphi \bcn \right\|_{L^2(\EkDom{h}{})}^2,
\end{equation}
and define $H^0_{h}$ to be the completion: $H^0_{h} (\Dom;\symmat) := \CL{H^{1}(\Dom;\symmat)}^{\| \cdot \|_{0,h}}$. 
Note that $H^0_{h} (\Dom;\symmat) \equiv L^2 (\Dom;\symmat) \oplus L^2 (\EkDom{h}{};\R)$, i.e., $\bvarphi \in H^0_{h} (\Dom;\symmat)$ is actually $\bvarphi \equiv (\bvarphiT,\varphinn)$, where $\bvarphiT \in L^2 (\Dom;\symmat)$ and $\varphinn \in L^2 (\EkDom{h}{})$, with no connection between $\bvarphiT$ and $\varphinn$.  We also have that $\bvarphi \in H^{1}(\Dom;\symmat) \subset H^0_{h} (\Dom;\symmat)$ implies $\bcn\tp \bvarphiT \bcn |_{\EkDom{h}{}} = \varphinn$ \cite{Babuska_MC1980}. 
Furthermore, we have a scalar valued function version of $\| \cdot \|_{0,h}$:
\begin{equation}\label{eqn:scalar_L2_mesh_dependent_norm_abstract}
\| v \|_{0,h}^2 := \| v \|_{L^2(\Dom)}^2 + h \| v \|_{L^2(\EkDom{h}{})}^2, \quad \text{for all } v \in H^1(\Dom),
\end{equation}
which satisfies the following estimate (proved in section SM1).%\cref{sec:proof_of_0_h_estimate}).
\begin{proposition}\label{prop:scalar_H0_h_norm_ineq}
For all $v \in H^1(\Dom)$, $\| v \|_{0,h}^2 \leq C \left( \| v \|_{L^2(\Dom)}^2 + h^{2} \| \grad{} v \|_{L^2(\Dom)}^2 \right)$, for some independent constant $C$.
\end{proposition}

Next, introduce the following skeleton subspaces
\begin{equation}\label{eqn:skeleton_spaces_general_bdy_cond}
\begin{split}
\SW_{h} &:= H^2_h(\Dom) \cap \zerobdy{H}^1(\Dom), \quad
\SV_{h} := \{ \bvarphi \in H^0_{h} (\Dom;\symmat) \mid \varphinn = 0 \text{ on } \Bsimsupp \},
\end{split}
\end{equation}
where $\SW_{h}$ is a mesh-dependent version of \cref{eqn:displacement_space_true} and $\SV_{h}$ is used for the stress $\bsigma$.  Note how essential and natural boundary conditions are imposed differently in \cref{eqn:skeleton_spaces_general_bdy_cond} than in \cref{eqn:displacement_space_true}.
In addition, we have the following Poincar\'{e} inequality, which follows by standard integration by parts arguments \cite{Blum_CM1990}.
\begin{proposition}\label{prop:scalar_H2_h_norm_trace_Poincare}
Define $\trinorm{v}_{2,h}^2 := \| \hess{} v \|_{L^2(\TkDom{h}{})}^2 + h^{-1}  \left\| \jump{\bcn \cdot \grad{} v} \right\|_{L^2(\EkIntDom{h}{})}^2$.  Then, $\trinorm{\cdot}_{2,h}$ is a norm on $\SW_{h}$. 
Moreover, there is a constant $\meshpoincare > 0$, depending only on $\Dom$, such that
\begin{equation}\label{eqn:scalar_H2_h_norm_trace_Poincare}
\begin{split}
	\| \grad{} v \|_{L^2(\Dom)} \leq \meshpoincare \trinorm{v}_{2,h}, \text{ for all } v \in \SW_{h}.
\end{split}
\end{equation}
\end{proposition}

%------------------------------------------------------------------------
\subsubsection{Mixed Skeleton Formulation}%\label{sec:}
%------------------------------------------------------------------------

Following \cite{Babuska_MC1980,Blum_CM1990}, we define a broken version of the Hessian operator.  Recalling the earlier discussion, we extend $\bfman{}{\bvarphi}{v}$ to all $\bvarphi \in H^0_{h} (\Dom;\symmat)$ and $v \in H^2_{h} (\Dom)$, i.e.
\begin{equation}\label{eqn:discrete_b_form}
\begin{split}
\bfman{}{\bvarphi}{v} 
&= -\inner{\bvarphiT}{\hess{} v}{\TkDom{h}{}} + \duality{\varphinn}{\jump{\bcn \cdot \grad{} v}}{\EkDom{h}{}},
\end{split}
\end{equation}
and extend $\afman{}{\bvarphi}{\bsigma}$ to all $\btau, \bvarphi \in H^0_{h} (\Dom;\symmat)$:
\begin{equation}\label{eqn:discrete_a_form}
\begin{split}
\afman{}{\btau}{\bvarphi} &= \sum_{T \in \TkDom{h}{}} \inner{\btau}{\binvelasten \bvarphi}{T} \equiv \inner{\btau}{\binvelasten \bvarphi}{\TkDom{h}{}}.
\end{split}
\end{equation}
Thus, we pose the following mixed weak formulation of the Kirchhoff plate problem.  Given $f \in H^{-1}(\Dom)$, find $\bsigma \in \SV_{h}$, $w \in \SW_{h}$ such that
\begin{equation}\label{eqn:Kirchhoff_manifold_skeleton_mixed}
\begin{split}
\afman{}{\bsigma}{\bvarphi} + \bfman{}{\bvarphi}{w} &= 0, \quad \forall \bvarphi \in \SV_{h}, \\
\bfman{}{\bsigma}{v} &= -\duality{f}{v}{}, ~ \forall v \in \SW_{h},
\end{split}
\end{equation}
where $\duality{\cdot}{\cdot}{}$ is the duality pairing between $H^{-1}$ and $\zerobdy{H}^{1}$.

\begin{remark}\label{rem:abstract_manifold_skeleton_mixed}
Assume for simplicity that the domain is smooth and $f \in H^{-1}$.  Then the solution  $(\bsigma,w)$ of \cref{eqn:plate_eqs} satisfies $w\in H^3$ and $\bsigma\in H^1$ and the pair $(\bsigma,w)$ solves \cref{eqn:Kirchhoff_manifold_skeleton_mixed}. In addition, any solution of \cref{eqn:Kirchhoff_manifold_skeleton_mixed} is also a solution of \cref{eqn:plate_eqs}, and so the formulations are equivalent. See \cite[Sec. 3]{Blum_CM1990} for details.

\end{remark}

%------------------------------------------------------------------------
\section{Curved Finite Elements}\label{sec:curved_FEM}
%------------------------------------------------------------------------

The basic theory of curved elements initiated in \cite{Ciarlet_CMAME1972} in two dimensions, with specific procedures for some low degree isoparametric Lagrange elements.  In  \cite{Zlamal_SJNA1974,Zlamal_IJNME1973,Zlamal_SJNA1973}, a theory for arbitrarily curved (two-dimensional) elements was given, while \cite{LRScott_PhD1973} gave a general procedure for arbitrary order isoparametric elements.  Later, \cite{Lenoir_SJNA1986} generalized the theory to any dimension and gave a method of constructing the curved elements.  The following sections give the essential parts of \cite{Lenoir_SJNA1986} that we need for this paper; section SM2 %\cref{sec:supplemental_curved_FEM}
gives a more complete review.

%------------------------------------------------------------------------
\subsection{Curved Triangulations}\label{sec:curved_triangulations}
%------------------------------------------------------------------------

We recall the parametric approach to approximating a domain with a curved boundary by a curvilinear triangulation $\TkDom{h}{m}$ of order $m\ge 1$, following \cite{Lenoir_SJNA1986}. The process begins with a conforming, shape-regular, straight-edged triangulation $\TkDom{h}{1}$ which triangulates a polygon $\Dom^1$ interpolating $\Dom$ (in the sense that the boundary vertices of $\Dom^1$ lie on the boundary of $\Dom$).  We define $\TkBdyDom{h}{1}$ to be the set of triangles with at least one vertex on the boundary. We make the following assumption.
\begin{hypothesis}\label{hyp:one_side_on_bdy}
Each triangle in $\TkDom{h}{1}$ has at most two vertices on the boundary and so has at most one edge contained in $\BdyDom^{1}$.
\end{hypothesis}
Next, for each $T^{1} \in \TkDom{h}{1}$, we define a map $\MapT{T}^{m}:T^{1}\to\R^2$ of polynomial degree $m$ which maps $T^{1}$ diffeomorphically onto a curvilinear triangle $T^{m}$. The map is determined by specifying the images of the Lagrange nodes of degree $m$ on $T^{1}$.  Nodes on an interior edge of $T^{1}$ are specified to remain fixed, while those on a boundary edge have their image determined by interpolation of a chart defining the boundary.  Nodes interior to $T^1$ are mapped in an intermediate fashion through their barycentric coordinates. See equation (14) of \cite{Lenoir_SJNA1986} for an explicit formula for $\MapT{T}^{m}\circ\widehat{\MapT{T}^{1}}$ where $\widehat{\MapT{T}^{1}}$ is the affine map from the standard reference triangle to $T^1$.  The maps $\MapT{T}^{m}$ so determined satisfy optimal bounds on their derivatives, as specified in \cite[Thm.~1 and 2]{Lenoir_SJNA1986}.  Moreover, the triangulation $\TkDom{h}{m}$ consisting of all the curvilinear triangles $T^m = \MapT{T}^{m}(T^1)$, $T^1 \in \TkDom{h}{1}$, is itself a conforming, shape regular triangulation that approximates $\Dom$ by $\Dom^{m} := \bigcup_{T^{m} \in \TkDom{h}{m}} T^{m}$.  We also denote by $\EkDom{h}{m}$ the set of edges of the triangulation $\TkDom{h}{m}$, which is partitioned into interior edges $\EkIntDom{h}{m}$ (all straight) and boundary edges $\EkBdyDom{h}{m}$ (possibly curved).  Thus $\BdyDom^{m} := \bigcup_{E^{m} \in \EkBdyDom{h}{m}} E^{m}$ is an $m$th order approximation of $\BdyDom$.
Note that, by construction, (i) $\MapT{T}^{1} \equiv \ident{T^{1}}$; (ii) if $T^{1}$ has no side on $\BdyDom$, then $\MapT{T}^{m} \equiv \ident{T^{1}}$; (iii) $\MapT{T}^{m} |_{E} = \ident{T^{1}} |_{E}$ for all interior edges $E \in \EkIntDom{h}{m}$.

Of course, the polynomial maps $\MapT{T}^{m}$ may be combined to define a piecewise polynomial diffeomorphism $\MapT{}^m:\Dom^1\to\Dom^m$.  Moreover, two of these maps, for degrees $l$ and $m$, may be combined to give a map between the corresponding approximate domains.  Referring to \cref{fig:mappings}(a), it is defined piecewise by
\begin{equation}\label{eqn:inter_curved_elem_map}
\begin{split}
\bPhi^{lm} |_{T} = \bPhi^{lm}_{T} &: T^{l} \to T^{m}, ~\text{ where }~ \bPhi^{lm}_{T} := \MapT{T}^{m} \circ (\MapT{T}^{l})^{-1}, \text{ so } \bPhi^{1m}_{T} \equiv \MapT{T}^{m}.
\end{split}
\end{equation}

\begin{figure}[ht]
\begin{center}
\centerline{\includegraphics[width=2.0in]{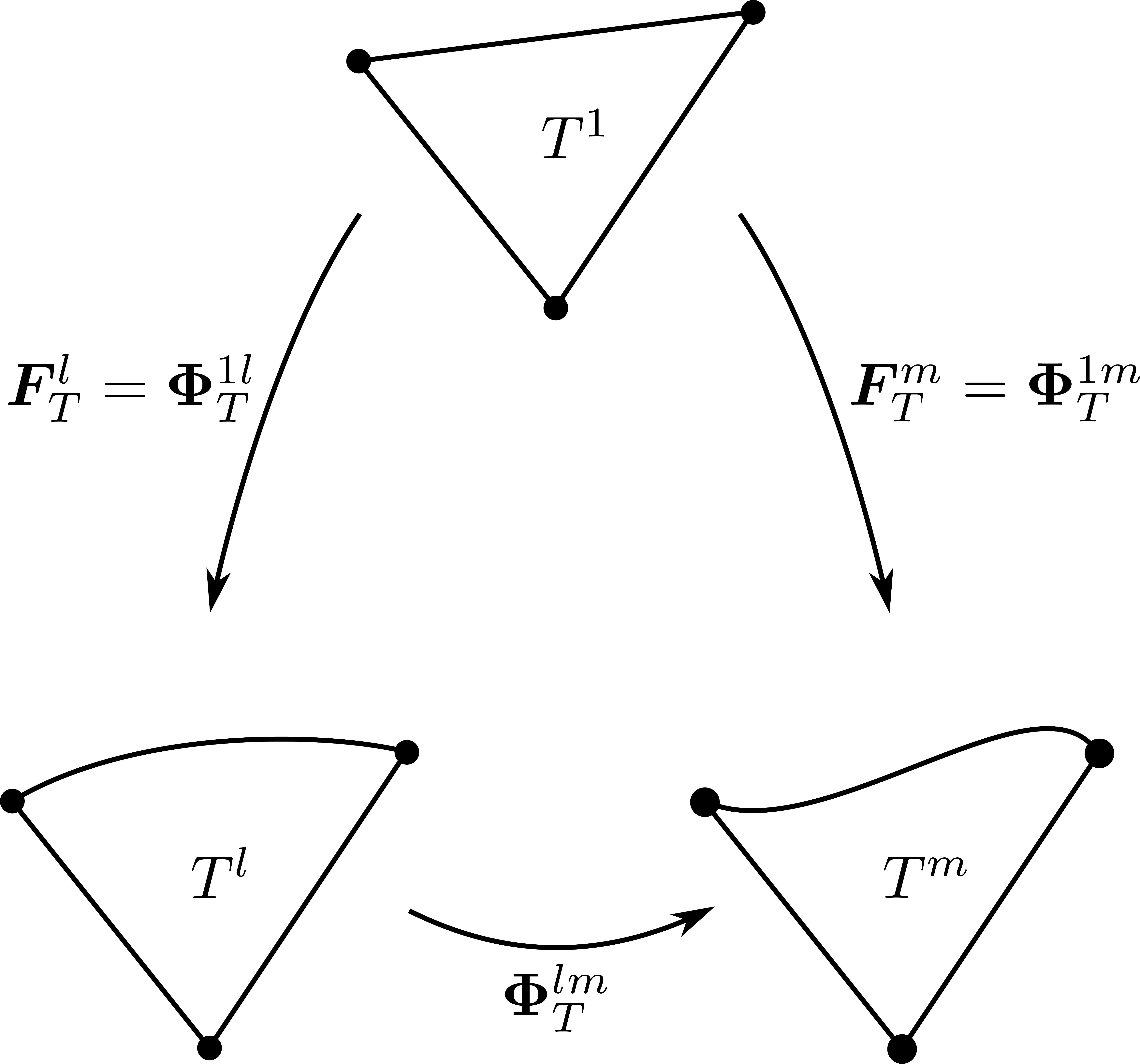}\qquad
\includegraphics[width=2.0in]{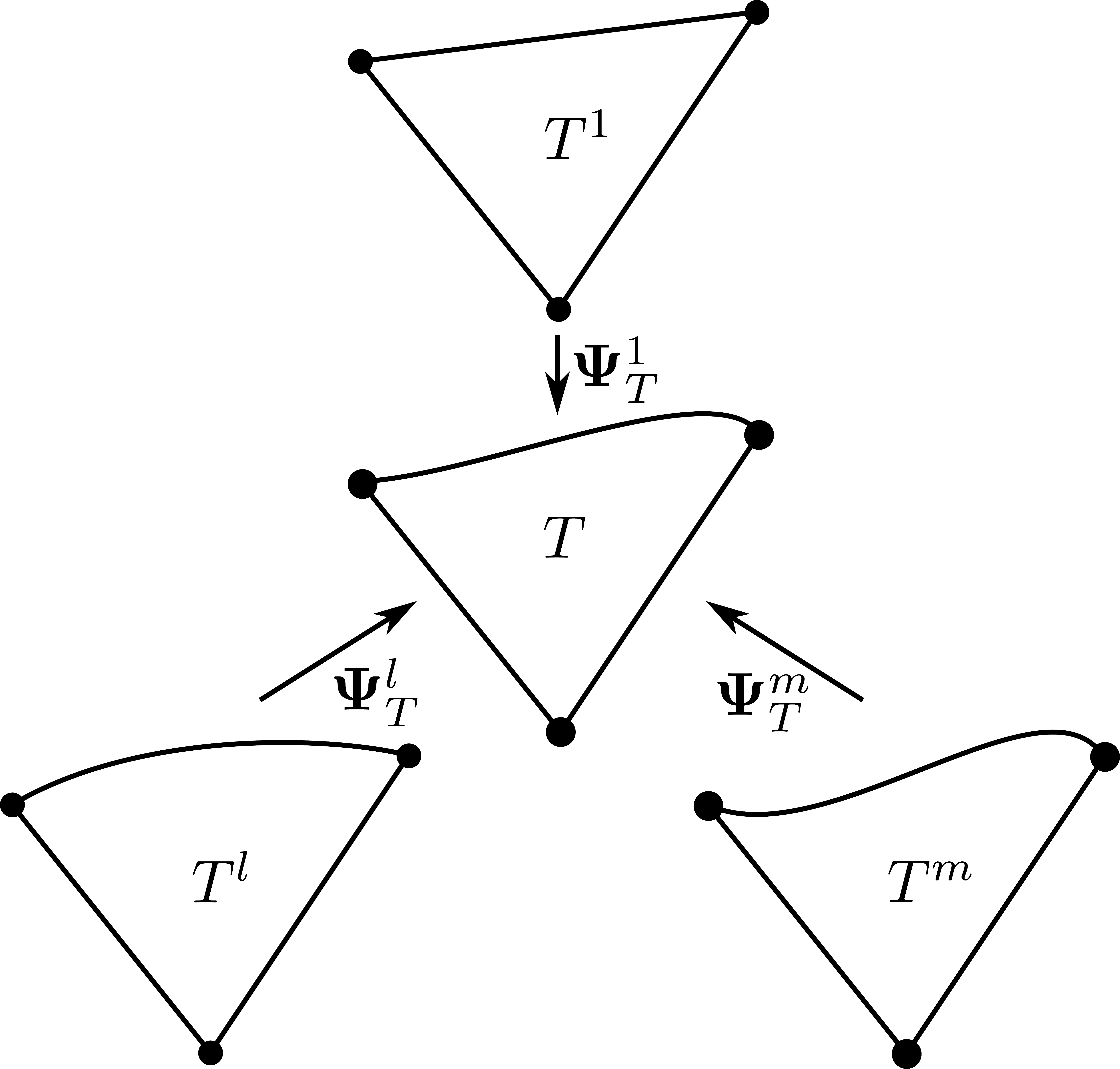}}
\caption{(a) Mappings between linear approximate triangles and approximate triangles of higher degree.  (b) Mappings between the approximate triangles and the exact curvilinear triangle.}
\label{fig:mappings}
\end{center}
\end{figure}

In order to compare the exact solution, defined on the exact domain $\Dom$, with
an approximation defined on the approximate domain $\Dom^{m}$, we require a map from the approximate domain to the true one. These can be defined element-wise in close analogy to $\MapT{T}^m$. Specifically, given a triangle $T^m\in\TkDom{h}{m}$ we define a map  $\bPsi_{T}^{m} : T^{m} \to \R^2$  which maps $T^m$ diffeomorphically onto a curvilinear triangle $T$ exactly fitting $\Dom$.  If $T^m$ has no boundary edges, the map is taken to be the identity.  Otherwise, $T^{m}$ has one edge $E^{m} \subset \BdyDom^{m}$, and the map is defined by \cite[eqn. (32)]{Lenoir_SJNA1986}. It restricts to the identity on the interior edges of $T^m$ and satisfies Propositions SM2.1 and SM2.2. %\cref{prop:Psi_map_estimate,prop:Psi_inverse_map_estimate}.
The curvilinear triangulation $\TkDom{h}{} := \{ \bPsi_{T}^{m} (T^{m}) \}_{T^{m} \in \TkDom{h}{m}}$ then exactly triangulates $\Dom$. The $\bPsi_{T}^{m}$ may be pieced together to give a global map $\bPsi^{m} : \Dom^{m} \to \Dom$.

We may view the exact domain and the corresponding triangulation as the limiting case of the approximate domain, and its triangulation, as $m\to\infty$.  This leads to alternative notations
$\Dom^{\infty} \equiv \Dom$, $\TkDom{h}{\infty} \equiv \TkDom{h}{}$,  $\bPhi^{l\infty} \equiv \bPsi^{l}$, $\MapT{T}^{\infty} \equiv \bPsi^{1}$, etc., which will sometimes be convenient.  Note that the use of the superscript infinity in the notation for these quantities is suggestive: the exact domain can be thought of as an infinite order approximation of itself.  However, this is merely a choice of notation.  We are not asserting here some sort of convergence of the polynomial approximate domains to the true domain.

Section SM2 %\Cref{sec:supplemental_curved_FEM}
gives further details on these maps, with the main results summarized in the next theorem (proved in
subsection SM2.3). %\cref{sec:proof_of_mapping_estimates}).
\begin{theorem}\label{thm:curved_elem_map_estimates}
Assume \cref{hyp:one_side_on_bdy}.  Then for all $1 \leq l \leq m \leq k$ and $m=\infty$, the maps $\MapT{T}^{m}$, $\MapT{T}^{l}$ described above satisfy
\begin{equation}\label{eqn:curved_elem_map_estim}
\begin{split}
\|\grad{}^s (\MapT{T}^{l} - \ident{T^{1}}) \|_{L^{\infty}(T^{1})} &\leq C h^{2 - s}, ~\text{ for }~ s = 0,1,2, \\
\|\grad{}^s (\MapT{T}^{m}  - \MapT{T}^{l}) \|_{L^{\infty}(T^{1})} &\leq C h^{l + 1 - s}, ~\text{ for }~ 0 \leq s \leq l+1, \\
1 - C h \leq \| [\grad{} \MapT{T}^{l}]^{-1} \|_{L^{\infty}(T^{1})} &\leq 1 + Ch, \quad \| [\grad{} \MapT{T}^{l}]^{-1} - \eye \|_{L^{\infty}(T^{1})} \leq C h,
\end{split}
\end{equation}
and the map $\bPhi^{lm}$ satisfies the estimates
\begin{equation}\label{eqn:curved_elem_map_optim}
\begin{split}
\|\grad{}^s (\bPhi^{lm}_{T} - \ident{T^{l}}) \|_{L^{\infty}(T^{l})} &\leq C h^{l+1-s}, ~\text{ for }~ 0 \leq s \leq l+1, \\
\|\grad{}^s ((\bPhi^{lm}_{T})^{-1} - \ident{T^{m}}) \|_{L^{\infty}(T^{m})} &\leq C h^{l+1-s}, ~\text{ for }~ 0 \leq s \leq l+1,
\end{split}
\end{equation}
\begin{equation}\label{eqn:curved_elem_map_estimate_isolate}
\begin{split}
%\left[ \bPhi^{lm} - \ident{T^{l}} \right] \circ \MapT{}^{l} &= \MapT{}^{m} - \MapT{}^{l}, \\
\left[ \grad{} (\bPhi^{lm} - \ident{T^{l}}) \right] \circ \MapT{}^{l} &= \grad{} (\MapT{}^{m} - \MapT{}^{l}) + O(h^{l + 1}), \\
\left[ \grad{}^2 (\bPhi^{lm} - \ident{T^{l}}) \cdot \tanb{\gamma} \right] \circ \MapT{}^{l} &= \grad{}^2 (\MapT{}^{m} - \MapT{}^{l}) \cdot \tanb{\gamma} + O(h^{l}), \text{ for } \gamma = 1, 2,
\end{split}
\end{equation}
where all constants depend on the piecewise $C^{k+1}$ norm of $\BdyDom$.
\end{theorem}
Analyzing the geometric error of the HHJ mixed formulation is delicate (recall \cref{sec:intro}).  Indeed, the identity \cref{eqn:curved_elem_map_estimate_isolate} will play an important role.

We close with a basic result relating norms on different order approximations of the same domain.  
The following result extends \cite[Thm 4.3.4]{Ciarlet_Book2002} to the mesh dependent norms in \cref{sec:skeleton_spaces}, and is proved in subsection SM2.4. %\cref{sec:equiv_norms}.
\begin{proposition}\label{prop:equiv_norms}
Assume the hypothesis of \cref{prop:scalar_H2_h_norm_trace_Poincare}.
Let $v \in H^2_h(\Dom^{m})$ and define $\hat{v} = v \circ \bPhi \in H^2_h(\Dom^{l})$, $\bPhi |_{T} := \bPhi_{T}^{lm}$ for any choice of $l, m \in \{ 1, 2, ..., k, \infty \}$.  Let $\| v \|_{2,h,m}$, $\| v \|_{0,h,m}$ denote the norms in \cref{eqn:scalar_H2_h_norm}, \cref{eqn:scalar_L2_mesh_dependent_norm_abstract} defined on $\Dom^{m}$.  Then, $\| \hess{} v \|_{L^2(\TkDom{h}{m})} \leq C \left( \| \hess{} \hat{v} \|_{L^2(\TkDom{h}{l})} + h^{l-1} \| \grad{} \hat{v} \|_{L^2(\TkDom{h}{l})} \right)$, and
\begin{equation}\label{eqn:equiv_norms}
\| v \|_{2,h,m} \leq C \left( \| \hat{v} \|_{2,h,l} + h^{l-1} \| \grad{} \hat{v} \|_{L^2(\Dom^{l})} \right), \quad \| v \|_{0,h,m} \approx C \| \hat{v} \|_{0,h,l},
\end{equation}
\begin{equation}\label{eqn:equiv_norms_zero_bdy}
	\| v \|_{2,h,m} \approx C \| \hat{v} \|_{2,h,l}, ~ \text{ if } v \in H^2_h(\Dom^{m}) \cap \zerobdy{H}^1 (\Dom^{m}),
\end{equation}
for some constant $C > 0$ depending on the domain, where we modify the norm subscript to indicate the order of the domain.
\end{proposition}

%------------------------------------------------------------------------
\subsection{Curved Lagrange Spaces}%\label{sec:}
%------------------------------------------------------------------------

Let $r$ be a positive integer and $m$ a positive integer or $\infty$. The (continuous) Lagrange finite element space of degree $r$ is defined on $\Dom^{m}$ via the mapping $\MapT{T}^m$:
\begin{equation}\label{eqn:std_CG_Lagrange_FE}
	\LAG_h^{m} \equiv \LAG_h^{m} (\Dom^{m}) := \{ v \in \zerobdy{H}^1 (\Dom^{m}) \mid v |_{T} \circ \MapT{T}^{m} \in \Pk_{r+1}(T^{1}), ~ \forall T \in \TkDom{h}{m} \}.
\end{equation}
For the case $m=\infty$ (the exact domain) we simply write $\LAG_h$.

If $v\in H^{2}_{h}(\Dom^{1})$, then, on each triangle $T^{1}$, $v$ is in $H^2$ and hence continuous up to the boundary of $T^{1}$.  Globally $v \in H^1(\Dom^1)$, and so has a well-defined trace on each edge.   
Consequently $v$ is continuous on $\CL{\Dom^{1}}$ and we can define the Lagrange interpolation operator $\Ilag_{h}^{1} : H^{2}_{h}(\Dom^{1}) \to \LAG_{h}^{1}$ \cite{Babuska_MC1980} defined on each element $T^{1} \in \TkDom{h}{1}$ by
\begin{equation}\label{eqn:scalar_space_interp_oper}
\begin{split}
(\Ilag_{h}^{1} v) (\uu) - v(\uu) &= 0, \quad \forall \text{ vertices } \uu \text{ of } T^{1}, \\
\int_{E^{1}} (\Ilag_{h}^{1} v - v) q \, \darc &= 0, \quad \forall q \in \Pk_{r-1}(E^{1}), ~\forall E^{1} \in \partial T^{1}, \\
\int_{T^{1}} (\Ilag_{h}^{1} v - v) q \, \darea &= 0, \quad \forall q \in \Pk_{r-2}(T^{1}).
\end{split}
\end{equation}
Then, given $v \in H^{2}_{h}(\Dom^{m})$, we define the global interpolation operator, $\Ilag_{h}^{m} : H^{2}_{h}(\Dom^{m}) \to \LAG_{h}^{m}$, element-wise through $\Ilag_{h}^{m} v \big{|}_{T^{m}} \circ \MapT{T}^{m} := \Ilag_{h}^{1} (v \circ \MapT{T}^{m})$.  Note that $v \circ \MapT{}^{m} \in C^{0}(\Dom^{1})$ because $v \in C^{0}(\Dom^{m})$ and $\MapT{}^{m}$ is continuous over $\Dom^{1}$.  Approximation results for $\Ilag_{h}^{m}$ are given in subsection SM3.2. %\cref{sec:lagrange_FE_approx_results}.
We also denote $\Ilag_{h}^{m,s}$ to be the above Lagrange interpolant on $\Dom^{m}$ onto continuous piecewise polynomials of degree $s$.  Thus, $\Ilag_{h}^{m,r+1} \equiv \Ilag_{h}^{m}$.

%------------------------------------------------------------------------
\section{The HHJ Method}\label{sec:manifold_HHJ}
%------------------------------------------------------------------------

We start with a space of tensor-valued functions, defined on curved domains, with special continuity properties, followed by a transformation rule for the forms in \cref{eqn:discrete_b_form,eqn:discrete_a_form}.  Next, we state the finite element approximation spaces for \cref{eqn:Kirchhoff_manifold_skeleton_mixed}, which conform to $H^0_{h} (\Dom^{m};\symmat)$ and $H^2_{h} (\Dom^{m})$, and define interpolation operators for these spaces while accounting for the effect of curved elements (recall that $1 \leq m \leq k$ or $m=\infty$).

%------------------------------------------------------------------------
\subsection{A Tensor Valued Space on Curved Domains}%\label{sec:}
%------------------------------------------------------------------------

For $p > 3/2$, let
\begin{equation}\label{eqn:HHJ_decomp_hessian_space}
\begin{split}
\Mnn^{m} (\Dom^{m}) := \{ \bvarphi \in L^2 (\Dom^{m};\symmat)  \mid \bvarphi |_{T^{m}}
   \in &W^{1,p}(T^{m};\symmat)\ \forall T^{m} \in \TkDom{h}{m},
   \\
   &\text{$\bvarphi$ normal-normal continuous}\}.
\end{split}
\end{equation}
Note that $\Mnn^{m} (\Dom^{m}) \subset H^0_{h} (\Dom^{m};\symmat)$ with $\varphinn \equiv \bcn\tp \bvarphiT \bcn$ on each mesh edge.

\begin{remark}\label{rem:assume_p_3/2}
The assumption that $p > 3/2$ is a technical simplification to ensure that the trace of a function in $\Mnn^{m} (\Dom^{m})$ onto the mesh skeleton $\EkDom{h}{m}$ is in $L^2(\EkDom{h}{m})$.
\end{remark}

In order to map between $\Mnn^{m} (\Dom^{m})$ and $\Mnn^{l} (\Dom^{l})$ (with $m \neq l$) such that normal-normal continuity is preserved, we need the following transformation rule.
\begin{definition}[Matrix Piola Transform]\label{defn:matrix_piola_transforms}
Let $\MapT{} : \widehat{\cD} \to \cD$ be an orientation-preserving diffeomorphism between domains in $\R^2$. Given $\bvarphi:\cD\to\symmat$, we define  its \emph{matrix Piola transform} $\hat{\bvarphi}:\widehat{\cD}\to \symmat$ by
\begin{equation}\label{eqn:matrix_Piola_transf_contra}
\hat{\bvarphi}(\hat{\vx}) = (\det \Jac{})^2 \Jac{}^{-1} \bvarphi(\vx)\Jac{}^{-T}
\end{equation}
where $\vx = \MapT{}(\hat{\vx})$, and $\Jac{} = \Jac{}(\hat{\vx}) = \grad{} \MapT{}(\hat{\vx})$.
\end{definition}
Note that \cref{eqn:matrix_Piola_transf_contra} is analogous to the Piola transform for $\Hdiv{\Dom}$ functions.

By elementary arguments, see (SM4.5), %\cref{eqn:matrix_Piola_transf_contra_nn_SUPP},
we find that
\begin{equation}\label{eqn:matrix_Piola_transf_contra_nn}
	\varphinn \circ \MapT{} = \hatvarphinn \, |(\grad{} \MapT{})\hat{\btv}|^{-2}.
\end{equation}
We shall apply the transform when the diffeomorphism is $\MapT{}^m$, which is piecewise smooth and continuous with respect to the mesh.  It follows that $(\grad{} \MapT{})\hat{\btv}$ is single-valued at interelement edges, so $\bvarphi$ is normal-normal continuous if and only if $\hat{\bvarphi}$ is.

We close with the following norm equivalences (see (SM4.7) and (SM4.8))
% \cref{eqn:mesh_depend_tensor_L2_norm_bnd_abstract,eqn:equiv_matrix_0_h_norm})
\begin{equation}\label{eqn:HHJ_tensor_norm_equiv}
   \| \bvarphi \|_{0,h,m} \approx \| \bvarphi \|_{L^2(\Dom^{m})}, ~ \forall \, \bvarphi \in \HHJ_{h}^{m}, \quad \| \bvarphi \|_{0,h,m} \approx \| \hat{\bvarphi} \|_{0,h,l}, ~ \forall \, \bvarphi \in H^0_{h}(\Dom^{m};\symmat),
\end{equation}
for all $1 \leq l,m \leq k, \infty$.

%------------------------------------------------------------------------
\subsection{Mapping Forms}\label{sec:mapping_forms}
%------------------------------------------------------------------------

The following result is crucial for analyzing the geometric error when approximating the solution on an approximate domain and also for deriving the discrete inf-sup condition on curved elements.  We define $\strip = \bigcup_{T \in \TkBdyDom{h}{}} T$ for the ``strip'' domain contained in $\Dom$.  In addition, we generalize the definitions \cref{eqn:discrete_a_form} and \cref{eqn:discrete_b_form} of the bilinear forms $\afman{}{\cdot}{\cdot}$ and $\bfman{}{\cdot}{\cdot}$ to include a superscript $m$ to indicate that they are defined on the approximate domain $\Dom^{m}$.
\begin{theorem}\label{thm:mapping_forms}
Let $1 \leq l \leq k$ such that $m > l$, recall $\bPhi \equiv \bPhi^{lm} : \Dom^{l} \to \Dom^{m}$ from \cref{eqn:inter_curved_elem_map}, for $1 < m \leq k$, and $m=\infty$, and set $\bJ := \grad{} \bPhi$. 
For all $\bsigma, \bvarphi \in \Mnn^{m}(\Dom^{m})$ and $v \in H^2_{h}(\Dom^{m})$, there holds
\begin{equation}\label{eqn:mapping_a_form}
\begin{split}
\afman{m}{\bsigma}{\bvarphi} &= \afman{l}{\hat{\bsigma}}{\hat{\bvarphi}} + O(h^{l}) \| \hat{\bsigma} \|_{L^2(\strip^{l})} \| \hat{\bvarphi} \|_{L^2(\strip^{l})},
\end{split}
\end{equation}
\begin{equation}\label{eqn:mapping_b_h_form}
\begin{split}
\bfman{m}{\bvarphi}{v} &= \bfman{l}{\hat{\bvarphi}}{\hat{v}} + \sum_{T^{l} \in \TkBdyDom{h}{l}} \duality{\hatvarphinn}{\hat{\bcn} \cdot \grad{} \left[ (\ident{T^{l}} - \bPhi_{T}) \cdot \CProj \grad{} \hat{v} \right]}{\partial T^{l}} \\
+ O(h^{l}&) \| \hat{\bvarphi} \|_{0,h,l} \| \grad{} \hat{v} \|_{H^1(\TkBdyDom{h}{l})} - \sum_{T^{l} \in \TkBdyDom{h}{l}} \inner{\hat{\bvarphi}}{\hess{} [(\ident{T^{l}} - \bPhi_{T}) \cdot \CProj \grad{} \hat{v}]}{T^{l}} \\
+ O(h^{l}&) \sum_{E^{l} \in \EkBdyDom{h}{l}} \| \hatvarphinn \|_{L^2(E^{l})} \| \grad{} \Ilag_{h}^{l,1} \hat{v} \|_{L^2(E^{l})},
\end{split}
\end{equation}
where $\bsigma, \bvarphi$ and $\hat{\bsigma}, \hat{\bvarphi}$ are related by the matrix Piola transform \cref{eqn:matrix_Piola_transf_contra} involving $\bPhi_{T}$, and $v |_{T} \circ \bPhi_{T} = \hat{v}$, $\Ilag_{h}^{l,1}$ is the Lagrange interpolation operator onto piecewise linears on $\Dom^{l}$, and $\CProj : L^2(\Dom^{l}) \to L^2(\Dom^{l})$ is the projection onto piecewise constants.
\end{theorem}
\begin{proof}
To derive \cref{eqn:mapping_a_form}, we use \cref{eqn:matrix_Piola_transf_contra} to obtain
\begin{equation}\label{eqn:a_form_mapped_exact_to_discrete}
\begin{split}
\afman{m}{\bsigma}{\bvarphi} &= \inner{\varphi^{\gamma \omega}}{\invelasten{\gamma}{\omega}{\alpha}{\beta} \sigma^{\alpha \beta}}{\Dom^{m}} = \sum_{T^{l} \in \TkDom{h}{l}} \inner{(\det \bJ_{T})^{-1} \hat{\varphi}^{\gamma \omega}}{\hatinvelasten{\gamma}{\omega}{\alpha}{\beta} \hat{\sigma}^{\alpha \beta}}{T^{l}},
\end{split}
\end{equation}
where $[\binvelasten]_{\gamma \omega \alpha \beta} \equiv \invelasten{\gamma}{\omega}{\alpha}{\beta} = 1/(2 \shearmod) \kron_{\gamma \alpha} \kron_{\omega \beta} - (\poisson/\youngmod) \kron_{\gamma \omega} \kron_{\alpha \beta}$, with $\kron_{\alpha \beta}$ being the Kronecker delta, $\hatinvelasten{\gamma}{\omega}{\alpha}{\beta} = 1/(2 \shearmod) \met{\gamma \alpha} \met{\omega \beta} - (\poisson/\youngmod) \met{\gamma \omega} \met{\alpha \beta}$, $\metmat := \bJ \tp \bJ$ is the induced metric, and $[\metmat]_{\alpha \beta} \equiv \met{\alpha \beta}$.  The result follows by adding and subtracting terms, noting that $\| \invelasten{\gamma}{\omega}{\alpha}{\beta} - \hatinvelasten{\gamma}{\omega}{\alpha}{\beta}\|_{L^{\infty}(\Dom^{l})} \leq C h^{l}$, and using that $\bPhi_{T} = \ident{T^{l}}$ for all $T^{l} \notin \TkBdyDom{h}{l}$.

As for \cref{eqn:mapping_b_h_form}, we start with \cref{eqn:discrete_b_form} and write it as
\begin{equation}\label{eqn:mapping_b_h_form_pf_1}
\begin{split}
\bfman{m}{\bvarphi}{v} &= - \sum_{T^{m} \in \TkDom{h}{m}} \left[ \inner{\bvarphiT}{\hess{} v}{T^{m}} - \duality{\varphinn}{\bcn \cdot \grad{} v}{\partial T^{m}} \right],
\end{split}
\end{equation}
noting that $\bcn\tp \bvarphiT \bcn \equiv \varphinn$.  It is only necessary to consider elements adjacent to the boundary, i.e., let $T^{m} \in \TkBdyDom{h}{m}$.  Then, mapping the first term in \cref{eqn:mapping_b_h_form_pf_1} from $\Dom^{m}$ to $\Dom^{l}$, we see that
\begin{equation}\label{eqn:mapping_b_h_form_pf_2}
\begin{split}
\inner{\bvarphi}{\hess{} v}{T^{m}} &= \inner{(\det \bJ)^{-1} \hat{\varphi}^{\alpha \beta}}{\left[ \pd{\alpha} \pd{\beta} \hat{v} - \pd{\gamma} \hat{v} \christ{\gamma}{\alpha \beta} \right]}{T^{l}},
\end{split}
\end{equation}
where $\christ{\gamma}{\alpha \beta}$ are the Christoffel symbols of the second kind (depending on the induced metric $\metmat$).
Note that $\christ{\gamma}{\alpha \beta} = \invmet{\mu \gamma} \pd{\alpha} \pd{\beta} (\bPhi \cdot \tanb{\mu})$, where $\tanb{\mu}$ is a canonical basis vector, and $\invmet{\mu \gamma} \equiv [\invmetmat]_{\mu \gamma}$ is the inverse metric.
Using the estimates in \cref{eqn:curved_elem_map_optim} for $\bPhi$, we can express \cref{eqn:mapping_b_h_form_pf_2} as
\begin{equation*}%\label{eqn:mapping_b_h_form_pf_4}
\begin{split}
\inner{\bvarphi}{\hess{} v}{T^{m}} &= \inner{\hat{\bvarphi}}{\grad{}^{2} \hat{v}}{T^{l}} + \inner{[(\det \bJ)^{-1} - 1] \hat{\varphi}^{\alpha \beta}}{\pd{\alpha} \pd{\beta} \hat{v}}{T^{l}} \\
&- \inner{\hat{\varphi}^{\alpha \beta}}{\pd{\gamma} \hat{v} \, \pd{\alpha} \pd{\beta} (\bPhi \cdot \tanb{\gamma})}{T^{l}} - \inner{\hat{\varphi}^{\alpha \beta}}{\pd{\gamma} \hat{v} (q^{\mu \gamma} - \kron^{\mu \gamma}) \pd{\alpha} \pd{\beta} (\bPhi \cdot \tanb{\mu})}{T^{l}} \\
&= \inner{\hat{\bvarphi}}{\grad{}^{2} \hat{v}}{T^{l}} - \inner{\hat{\varphi}^{\alpha \beta} \grad{} \hat{v}}{\pd{\alpha} \pd{\beta} \bPhi}{T^{l}} + O(h^{l}) \| \hat{\bvarphi} \|_{L^2(T^{l})} \| \grad{} \hat{v} \|_{H^1(T^{l})},
\end{split}
\end{equation*}
where we introduced $q^{\mu \gamma} = (\det \bJ)^{-1} \invmet{\mu \gamma}$, and note that $\| q^{\mu \gamma} - \kron^{\mu \gamma} \|_{L^{\infty}(T^{l})} \leq C h^{l}$ for all $T^{l} \in \TkDom{h}{l}$.  Furthermore, using the piecewise projection $\CProj |_{T^{l}} : L^2(T^{l}) \to \R$ onto constants, we have that
\begin{equation}\label{eqn:mapping_b_h_form_pf_4}
\begin{split}
	\inner{\hat{\varphi}^{\alpha \beta} \grad{} \hat{v}}{\pd{\alpha} \pd{\beta} \bPhi}{T^{l}} &= \inner{\hat{\varphi}^{\alpha \beta} \CProj \grad{} \hat{v}}{\pd{\alpha} \pd{\beta} \bPhi}{T^{l}} \\
	& + \inner{\hat{\varphi}^{\alpha \beta} [\grad{} \hat{v} - \CProj \grad{} \hat{v}]}{\pd{\alpha} \pd{\beta} (\bPhi - \ident{T^{l}})}{T^{l}} \\
	&\leq \inner{\hat{\varphi}^{\alpha \beta} \CProj \grad{} \hat{v}}{\pd{\alpha} \pd{\beta} \bPhi}{T^{l}} + C h^{l} \| \hat{\bvarphi} \|_{L^2(T^{l})} \| \hess{} \hat{v} \|_{L^2(T^{l})},
\end{split}
\end{equation}
further noting that $\inner{\hat{\varphi}^{\alpha \beta} \CProj \grad{} \hat{v}}{\pd{\alpha} \pd{\beta} \bPhi}{T^{l}} = \inner{\hat{\bvarphi}}{\hess{} [(\bPhi - \ident{T^{l}}) \cdot \CProj \grad{} \hat{v}]}{T^{l}}$.

Next, consider the second term in \cref{eqn:mapping_b_h_form_pf_1}. Express $\partial T^{m} =: E_{1}^{m} \cup E_{2}^{m} \cup \widetilde{E}^{m}$, where $\widetilde{E}^{m}$ is the curved side, and map from $\partial T^{m}$ to $\partial T^{l}$:
\begin{equation}\label{eqn:mapping_b_h_form_pf_5}
\begin{split}
	\duality{\varphinn}{\bcn \cdot \grad{} v}{\partial T^{m}} &= \duality{\frac{\det \bJ}{|\bJ \hat{\btv}|^2} \hatvarphinn}{\hat{\bcn} \cdot \metmat^{-1} \grad{} \hat{v}}{\widetilde{E}^{l}} + \duality{\hatvarphinn}{\hat{\bcn} \cdot \bJ\invtp \grad{} \hat{v}}{E_{1}^{l} \cup E_{2}^{l}},
\end{split}
\end{equation}
where $\hat{\bcn}$ is the unit normal on $\partial T^{l}$, and we used \cref{eqn:matrix_Piola_transf_contra_nn}.  Mapping the non-curved edges is simpler because $\bPhi_{T} = \ident{T^{l}}$ on $E_{1}^{l} \cup E_{2}^{l}$, so $E^{m} \equiv E^{l}$ and $\bcn \equiv \hat{\bcn}$. 
For convenience, define $\bR = \bJ^{-1} (\det \bJ) |\bJ \hat{\btv}|^{-2}$, which implies, by \cref{eqn:curved_elem_map_optim}, that $\| \bR - \eye_{2 \times 2} \|_{L^\infty(\widetilde{E}^{l})} \leq C_1 h^{l}$.  Since $\invmetmat = \bJ^{-1} \bJ\invtp$, we get
\begin{equation}\label{eqn:mapping_b_h_form_pf_6}
\begin{split}
	&\duality{\varphinn}{\bcn \cdot \grad{} v}{\partial T^{m}} = \duality{\hatvarphinn}{\hat{\bcn} \cdot \bJ\invtp \grad{} \hat{v}}{\partial T^{l}} + \| \hat{\bcn} \cdot [\bR - \eye] \bJ\invtp \|_{L^{\infty}(\widetilde{E}^{l})} \times \\
	&\qquad\qquad\quad \times \| \hatvarphinn \|_{L^2(\widetilde{E}^{l})} \left[ \| \grad{} \hat{v} - \grad{} \Ilag_{h}^{l,1} \hat{v} \|_{L^2(\widetilde{E}^{l})} + \| \grad{} \Ilag_{h}^{l,1} \hat{v} \|_{L^2(\widetilde{E}^{l})} \right] \\
	&\quad\qquad= \duality{\hatvarphinn}{\hat{\bcn} \cdot \bJ\invtp \grad{} \hat{v}}{\partial T^{l}} + O(h^{l}) h^{1/2} \| \hatvarphinn \|_{L^2(\widetilde{E}^{l})} \| \hess{} \hat{v} \|_{L^2(T^{l})} \\
	&\qquad\qquad\qquad + O(h^{l}) \| \hatvarphinn \|_{L^2(\widetilde{E}^{l})} \| \grad{} \Ilag_{h}^{l,1} \hat{v} \|_{L^2(\widetilde{E}^{l})},
\end{split}
\end{equation}
for all $T^{l} \in \TkBdyDom{h}{l}$, where $\Ilag_{h}^{l,1}$ satisfies, by \cref{eqn:scaling_trace_estimate}, (SM3.6), and (SM3.7),%eqn:Lagrange_interp_operator_H^0_h,eqn:Lagrange_interp_operator_H^2_h},
\begin{equation}\label{eqn:mapping_b_h_form_trace_estim}
\begin{split}
	h^{1/2} & \| \grad{} \hat{v} - \grad{} \Ilag_{h}^{l,1}\hat{v} \|_{L^2(\partial T^{l})} \leq \\
	&C \left( \| \grad{} (\hat{v} -  \Ilag_{h}^{l,1} \hat{v}) \|_{L^2(T^{l})} + h \| \hess{} (\hat{v} -  \Ilag_{h}^{l,1} \hat{v}) \|_{L^2(T^{l})} \right) \leq C h \| \hess{} \hat{v} \|_{L^2(T^{l})}.
\end{split}
\end{equation}

Furthermore,
\begin{equation*}
\duality{\hatvarphinn}{\hat{\bcn} \cdot \bJ\invtp \grad{} \hat{v}}{\partial T^{l}} = \duality{\hatvarphinn}{\hat{\bcn} \cdot [\bJ\invtp - \eye] \grad{} \hat{v}}{\partial T^{l}} + \duality{\hatvarphinn}{\hat{\bcn} \cdot \grad{} \hat{v}}{\partial T^{l}},
\end{equation*}
and expanding further, and using \cref{eqn:curved_elem_map_optim}, gives
\begin{equation}\label{eqn:mapping_b_h_form_pf_8}
\begin{split}
	&\duality{\hatvarphinn}{\hat{\bcn} \cdot [\bJ\invtp - \eye] \grad{} \hat{v}}{\partial T^{l}} = \duality{\hatvarphinn}{\hat{\bcn} \cdot \bJ\invtp [\eye - \bJ\tp] \grad{} \hat{v}}{\partial T^{l}} \\
	&\quad = \duality{\hatvarphinn}{\hat{\bcn} \cdot [\eye - \bJ\tp] \grad{} \hat{v}}{\partial T^{l}} + \duality{\hatvarphinn}{\hat{\bcn} \cdot [\bJ\invtp - \eye] [\eye - \bJ\tp] \grad{} \hat{v}}{\partial T^{l}} \\
	&\quad = \duality{\hatvarphinn}{\hat{\bcn} \cdot [\eye - \bJ\tp] (\CProj \grad{} \hat{v})}{\partial T^{l}} + \duality{\hatvarphinn}{\hat{\bcn} \cdot [\bJ\invtp - \eye] [\eye - \bJ\tp] \grad{} \hat{v}}{\partial T^{l}} \\
	&\quad + \duality{\hatvarphinn}{\hat{\bcn} \cdot [\eye - \bJ\tp] (\grad{} \hat{v} - \CProj \grad{} \hat{v})}{\partial T^{l}} \\
	&\leq \duality{\hatvarphinn}{\hat{\bcn} \cdot [\eye - \bJ\tp] (\CProj \grad{} \hat{v})}{\partial T^{l}} + C h^{l} h^{1/2} \| \hatvarphinn \|_{L^2(\partial T^{l})} h^{1/2} \| \grad{} \hat{v} \|_{L^2(\partial T^{l})} \\
	&\qquad + C h^{l-1} h^{1/2}\| \hatvarphinn \|_{L^2(\partial T^{l})}  h^{1/2} \| \grad{} \hat{v} - \CProj \grad{} \hat{v} \|_{L^2(\partial T^{l})}.
\end{split}
\end{equation}
Combining with \cref{eqn:mapping_b_h_form_pf_8}, noting that $\CProj$ satisfies an estimate similar to \cref{eqn:mapping_b_h_form_trace_estim}, and using \cref{eqn:scaling_trace_estimate} again, we obtain
\begin{equation}\label{eqn:mapping_b_h_form_pf_9}
\begin{split}
\duality{\hatvarphinn}{\hat{\bcn} \cdot [\bJ\invtp - \eye] \grad{} \hat{v}}{\partial T^{l}} &= \duality{\hatvarphinn}{\hat{\bcn} \cdot \grad{} \left[ (\ident{T^{l}} - \bPhi) \cdot (\CProj \grad{} \hat{v}) \right]}{\partial T^{l}} \\
&\quad + O(h^{l}) \left( h^{1/2} \| \hatvarphinn \|_{L^2(\partial T^{l})} \right) \| \grad{} \hat{v} \|_{H^1(T^{l})}.
\end{split}
\end{equation}
Combining the above results and summing over all $T^{m} \in \TkBdyDom{h}{m}$ completes the proof.
\end{proof}
A simple consequence of \cref{thm:mapping_forms} is
\begin{equation}\label{eqn:mapping_b_h_form_part_2}
\begin{split}
	\bfman{m}{\bvarphi}{v} &= \bfman{l}{\hat{\bvarphi}}{\hat{v}} + O(h^{l-1}) \| \hat{\bvarphi} \|_{0,h,l} \| \grad{} \hat{v} \|_{2,h,l}.
\end{split}
\end{equation}

%------------------------------------------------------------------------
\subsection{The HHJ Curved Finite Element Space}\label{sec:HHJ_FE_Space}
%------------------------------------------------------------------------

We can use \cref{eqn:matrix_Piola_transf_contra} to build the global, conforming, HHJ finite element space (on curved elements) by mapping from a reference element (see subsection SM4.2 %\cref{sec:HHJ_matrix_Piola}
for details), i.e., $\HHJ_{h}^{m} \equiv \HHJ_{h}^{m} (\Dom^{m}) \subset \Mnn^{m} (\Dom^{m})$ is defined by
\begin{equation}\label{eqn:discrete_tensor_nn_space}
\begin{split}
\HHJ_{h}^{m} (\Dom^{m}) := \{ \bvarphi \in \Mnn^{m} & (\Dom^{m}) \mid \bvarphi |_{T} \circ \MapT{T}^{m} := (\det \grad{} \MapT{T}^{m})^{-2} (\grad{} \MapT{T}^{m}) \hat{\bvarphi} (\grad{} \MapT{T}^{m})\tp, \\
&\qquad\qquad\qquad \hat{\bvarphi} \in \Pk_{r}(T^{1};\symmat), ~ \forall T^{m} \in \TkDom{h}{m} \}.
\end{split}
\end{equation}
Note that $\HHJ_{h}^{m}$ is isomorphic to $\HHJ_{h}^{1}$, for $1 \leq m \leq k$ and $m=\infty$. 

We also have the following tensor-valued interpolation operator $\IHHJ_{h}^{1} : \Mnn^{1} (\Dom^{1}) \to \HHJ_{h}^{1}$ \cite{Brezzi_RIA1976,Babuska_MC1980} defined on each element $T^{1} \in \TkDom{h}{1}$ by
\begin{equation}\label{eqn:HHJ_hessian_interp_oper_T1}
\begin{split}
\int_{E^{1}} \bcn\tp \left[ \IHHJ_{h}^{1} \bvarphi - \bvarphi \right] \bcn \, q \, \darc &= 0, \quad \forall q \in \Pk_{r}(E^{1}), ~\forall E^{1} \in \partial T^{1}, \\
\int_{T^{1}} \left[ \IHHJ_{h}^{1} \bvarphi - \bvarphi \right] \dd \bmeta \, \darea &= 0, \quad \forall \bmeta \in \Pk_{r-1}(T^{1};\symmat).
\end{split}
\end{equation}
Recall \cref{thm:mapping_forms} and set $\bPhi_{T} \equiv \bPhi^{1m}_{T} := \MapT{T}^{m} : T^{1} \to T^{m}$ with $\bJ_{T} := \grad{} \bPhi_{T}$. Now, given $\bvarphi \in \Mnn^{m} (\Dom^{m})$, we define the global interpolation operator, $\IHHJ_{h}^{m} : \Mnn^{m} (\Dom^{m}) \to \HHJ_{h}^{m}$, element-wise through
\begin{equation}\label{eqn:HHJ_hessian_interp_oper}
\begin{split}
\IHHJ_{h}^{m} \bvarphi \big{|}_{T^{m}} \circ \bPhi_{T} &= (\det \bJ_{T})^{-2} \bJ_{T} \left( \IHHJ_{h}^{1} \hat{\bvarphi} \right) \bJ_{T}\tp,
\end{split}
\end{equation}
where $\hat{\bvarphi} := (\det \bJ_{T})^{2} \bJ_{T}^{-1} (\bvarphi \circ \bPhi_{T}) \bJ_{T}\invtp$ (i.e., see \cref{eqn:matrix_Piola_transf_contra}).  The operator $\IHHJ_{h}^{m}$ satisfies many basic approximation results which can be found in subsection SM4.3. %\cref{sec:HHJ_approx_results}.

On affine elements, we have a Fortin like property involving $\bfman{1}{\cdot}{\cdot}$ \cite{Brezzi_RIA1976,Babuska_MC1980,Blum_CM1990}:
\begin{equation}\label{eqn:b_h_form_interp_oper_Fortin}
\begin{split}
\bfman{1}{\bvarphi - \IHHJ_{h}^{1} \bvarphi}{\theta_{h} v_h} &= 0, \quad \forall \bvarphi \in \Mnn^{1}(\Dom^{1}), \quad v_h \in \LAG_h^{1}, \\
\bfman{1}{\theta_{h} \bvarphi_{h}}{v - \Ilag_{h}^{1} v} &= 0, \quad \forall \bvarphi_{h} \in \HHJ_h^{1}, \quad v \in H^2_{h}(\Dom^{1}),
\end{split}
\end{equation}
which holds for \emph{any} piecewise constant function $\theta_{h}$ defined on $\TkDom{h}{1}$; in \cite{Brezzi_RIA1976,Babuska_MC1980,Blum_CM1990}, it is assumed that $\theta_{h} \equiv 1$.
However, \cref{eqn:b_h_form_interp_oper_Fortin} does not hold on curved elements, but instead we have the following result.

\begin{lemma}\label{lem:b_h_form_curved_Fortin_oper}
Let $1 \leq m \leq k$, or $m=\infty$, and set $r \geq 0$ to be the degree of HHJ space $\HHJ_{h}^{m}$, and $r+1$ to be the degree of the Lagrange space $\LAG_{h}^{m}$.  Moreover, assume $\HHJ_{h}^{m}$ and $\LAG_{h}^{m}$ impose \emph{no} boundary conditions. Then, the following estimates hold:
\begin{equation}\label{eqn:b_h_form_curved_Fortin_oper}
\begin{split}
\left| \bfman{m}{\bvarphi_h}{v - \Ilag_{h}^{m} v} \right| &\leq C \| \bvarphi_{h} \|_{L^2(\strip^{m})}
 \\
& \times \left( \| \grad{} (v - \Ilag_{h}^{m} v) \|_{L^2(\strip^{m})} + h \| \grad{}^2 (v - \Ilag_{h}^{m} v) \|_{L^2(\TkBdyDom{h}{m})} \right), \\
\left| \bfman{m}{\bvarphi - \IHHJ_{h}^{m} \bvarphi}{v_h} \right| &\leq C \| \bvarphi -\IHHJ_{h}^{m} \bvarphi \|_{H^{0}_{h}(\strip^{m})} \| \grad{} v_{h} \|_{L^2(\strip^{m})},
\end{split}
\end{equation}
for all $\bvarphi \in \Mnn(\Dom^{m})$, $v_h \in \LAG_{h}^{m}$, and all $\bvarphi_{h} \in \HHJ_h^{m}$, $v \in H^2_{h}(\Dom^{m})$, where $C$ is an independent constant.  Note that $C=0$ if $m=1$.
\end{lemma}
\begin{proof}

Consider the map $\bPhi_{T} : T^{l} \to T^{m}$ defined in \cref{thm:mapping_forms} with approximation properties given by \cref{eqn:curved_elem_map_optim}.
Let $v \in H^2_{h}(\Dom^{m})$, so by definition, there exists $\hat{v} \in H^2_{h}(\Dom^{l})$ such that $v \circ \bPhi_{T} = \hat{v}$ on each $T^{l} \in \TkDom{h}{l}$. By \cref{prop:equiv_norms}, $\| \hess{} v \|_{L^2(\TkDom{h}{m})} \approx C \left( \| \hess{} \hat{v} \|_{L^2(\TkDom{h}{l})} + h^{l-1} \| \grad{} \hat{v} \|_{L^2(\TkDom{h}{l})} \right)$.  Moreover, given $\bvarphi \in \Mnn(\Dom^{m})$ there exists $\hat{\bvarphi} \in \Mnn(\Dom^{l})$ given by
$\bvarphi \circ \bPhi_{T} := (\det \grad{} \bPhi_{T})^{-2} (\grad{} \bPhi_{T}) \hat{\bvarphi} (\grad{} \bPhi_{T})\tp$ (c.f. \cref{eqn:matrix_Piola_transf_contra}).
%$\bvarphi \circ (\bPhi_{T}) = \PIOLA{\bPhi_{T}} \dd \hat{\bvarphi}$.  
By \cref{eqn:HHJ_tensor_norm_equiv} $\| \bvarphi \|_{0,h,m} \approx \| \hat{\bvarphi} \|_{0,h,l}$.

Next, we estimate the ``problematic'' terms in \cref{eqn:mapping_b_h_form}.  First, we have
\begin{equation}\label{eqn:b_h_form_curved_Fortin_oper_pf_1}
\begin{split}
	\inner{\hat{\bvarphi}}{\hess{} [(\ident{T^{l}} - \bPhi_{T}) \cdot \CProj \grad{} \hat{v}]}{T^{l}} &\leq C h^{l-1} \|  \hat{\bvarphi} \|_{L^2(T^{l})} \| \grad{} \hat{v} \|_{L^{2}(T^{l})},
\end{split}
\end{equation}
where we used the optimal mapping properties in \cref{eqn:curved_elem_map_optim}.  In addition, we have
\begin{equation}\label{eqn:b_h_form_curved_Fortin_oper_pf_2}
\begin{split}
\duality{\hatvarphinn}{\hat{\bcn} \cdot \grad{} \left[ (\ident{T^{l}} - \bPhi_{T}) \cdot \CProj \grad{} \hat{v} \right]}{\partial T^{l}} \leq C h^{l} &\| \hatvarphinn \|_{L^2(\partial T^{l})} \| \CProj \grad{} \hat{v} \|_{L^2(\partial T^{l})} \\
\leq C h^{l-1} & h^{1/2} \| \hatvarphinn \|_{L^2(\partial T^{l})} \| \grad{} \hat{v} \|_{L^2(T^{l})},
\end{split}
\end{equation}
by the inverse estimate $\| \CProj \grad{} \hat{v} \|_{L^2(\partial T^{l})} \leq C h^{-1/2} \| \CProj \grad{} \hat{v} \|_{L^2(T^{l})}$.  Lastly,
\begin{equation}\label{eqn:b_h_form_curved_Fortin_oper_pf_3}
\begin{split}
	O(h^{l}) \| \hatvarphinn \|_{L^2(E^{l})} \| \grad{} \Ilag_{h}^{l,1} \hat{v} \|_{L^2(E^{l})} \leq C h^{l-1} & h^{1/2} \| \hatvarphinn \|_{L^2(E^{l})} \| \grad{} \hat{v} \|_{L^2(T^{l})},
\end{split}
\end{equation}
for all $E^{l} \in \EkBdyDom{h}{l}$ (again using an inverse estimate and stability of the interpolant).  Plugging \cref{eqn:b_h_form_curved_Fortin_oper_pf_1,eqn:b_h_form_curved_Fortin_oper_pf_2,eqn:b_h_form_curved_Fortin_oper_pf_3} into \cref{eqn:mapping_b_h_form}, yields
\begin{equation}\label{eqn:b_h_form_curved_Fortin_oper_pf_5}
\begin{split}
\left| \bfman{m}{\bvarphi}{v} \right| &\leq \left| \bfman{l}{\hat{\bvarphi}}{\hat{v}} \right| + C h^{l-1} \| \hat{\bvarphi} \|_{0,h,l} \left( \| \grad{} \hat{v} \|_{L^2(\strip^{l})} + h \| \grad{}^2 \hat{v} \|_{L^2(\TkBdyDom{h}{l})} \right).
\end{split}
\end{equation}

For the first estimate in \cref{eqn:b_h_form_curved_Fortin_oper}, set $l=1$,
replace $v$ with $v - \Ilag_{h}^{m} v$, set $\bvarphi = \bvarphi_h \in \HHJ_h^{m}$, use \cref{eqn:b_h_form_interp_oper_Fortin}, and \cref{eqn:HHJ_tensor_norm_equiv} to get
\begin{equation}\label{eqn:b_h_form_curved_Fortin_oper_pf_6}
\begin{split}
\left| \bfman{m}{\bvarphi_h}{v - \Ilag_{h}^{m} v} \right| &\leq C \| \hat{\bvarphi}_h \|_{L^2(\strip^{1})} \left( \| \grad{} (\hat{v} - \Ilag_{h}^{1} \hat{v}) \|_{L^2(\strip^{1})} + h \| \grad{}^2 (\hat{v} - \Ilag_{h}^{1} \hat{v}) \|_{L^2(\TkBdyDom{h}{1})} \right),
\end{split}
\end{equation}
then use equivalence of norms (see \cref{prop:equiv_norms,eqn:HHJ_tensor_norm_equiv}).

For the second estimate, replace $\bvarphi$ with $\bvarphi - \IHHJ_{h}^{m} \bvarphi$, set $v = v_h \in \LAG_{h}^{m}$, use \cref{eqn:b_h_form_interp_oper_Fortin}, and use an inverse inequality to get
\begin{equation}\label{eqn:b_h_form_curved_Fortin_oper_pf_7}
\begin{split}
\left| \bfman{m}{\bvarphi - \IHHJ_{h}^{m} \bvarphi}{v_h} \right| &\leq C \| \hat{\bvarphi} -\IHHJ_{h}^{1} \hat{\bvarphi} \|_{0,h,1} \| \grad{} \hat{v}_h \|_{L^2(\strip^{1})},
\end{split}
\end{equation}
followed by an equivalence of norms argument.
\end{proof}

%------------------------------------------------------------------------
\subsection{The HHJ Mixed Formulation}%\label{sec:}
%------------------------------------------------------------------------

We pose \cref{eqn:Kirchhoff_manifold_skeleton_mixed} on $\Dom^{m}$ with continuous skeleton spaces denoted $\SV^{m}_{h} \equiv \SV^{m}_{h} (\Dom^{m})$ and $\SW^{m}_{h} \equiv \SW^{m}_{h} (\Dom^{m})$. Fixing the polynomial degree $r \geq 0$, the conforming finite element spaces are
\begin{equation}\label{eqn:HHJ_discrete_spaces}
\V_h^{m} \subset \SV^{m}_{h}, \qquad \W_h^{m} \subset \SW^{m}_{h}.
\end{equation}
The conforming finite element approximation to \cref{eqn:Kirchhoff_manifold_skeleton_mixed} is as follows.
Given $f \in H^{-1}(\Dom^{m})$, find $\bsigma_{h} \in \V_h^{m}$, $w_h \in \W_h^{m}$ such that
\begin{equation}\label{eqn:Kirchhoff_manifold_HHJ_mixed}
\begin{split}
\afman{m}{\bsigma_{h}}{\bvarphi} + \bfman{m}{\bvarphi}{w_h} &= 0, ~ \forall \bvarphi \in \V_h^{m}, \\
\bfman{m}{\bsigma_{h}}{v} &= -\duality{f}{v}{\Dom^{m}}, ~ \forall v \in \W_h^{m}.
\end{split}
\end{equation}
The well-posedness of \cref{eqn:Kirchhoff_manifold_HHJ_mixed} is established in the next section, i.e., we prove the classic LBB conditions \cite{Boffi_book2013}.  With this, we have the following a priori estimate:
\begin{equation}\label{eqn:Kirchhoff_manifold_HHJ_mixed_a_priori_est}
\| w_h \|_{2,h,m} + \| \bsigma_h \|_{0,h,m} \leq C \| f \|_{H^{-1}(\Dom^{m})}.
\end{equation}
Note that LBB conditions for \cref{eqn:Kirchhoff_manifold_HHJ_mixed}, for the case $m=1$, was originally shown in \cite{Blum_CM1990}.

%------------------------------------------------------------------------
\subsubsection{Well-posedness}%\label{sec:}
%------------------------------------------------------------------------

Obviously, we have
\begin{equation}\label{eqn:bounded_a_form}
\afman{m}{\bsigma}{\bvarphi} \leq \afcont \| \bsigma \|_{L^2(\Dom^{m})} \| \bvarphi \|_{L^2(\Dom^{m})}, \quad \forall \bsigma, \bvarphi \in H^0_{h} (\Dom^{m};\symmat) \supset \HHJ_h^{m},
\end{equation}
\begin{equation}\label{eqn:bounded_b_form}
\begin{split}
|\bfman{m}{\bvarphi}{v}| &\leq \bfcont \| \bvarphi \|_{0,h,m} \| v \|_{2,h,m}, \quad \forall \bvarphi \in \V_{h}^{m}, ~  v \in \W_{h}^{m},
\end{split}
\end{equation}
and we have coercivity of $\afman{m}{\cdot}{\cdot}$, which is a curved element version of \cite[Thm. 2]{Babuska_MC1980}. 
\begin{lemma}\label{lem:mesh_depend_a_coercive_over_kernel_b_abstract}
Assume the domain $\Dom^{m} \subset \R^2$ is piecewise smooth consisting of curved elements as described in \cref{sec:curved_FEM}.
Then there is a constant $\afcoer > 0$, independent of $h$ and $m$, such that
\begin{equation}\label{eqn:a_coercive_over_kernel_b_h}
\begin{split}
\afman{m}{\bsigma}{\bsigma} &\geq \min (|\binvelasten|) \| \bsigma \|_{L^2(\Dom^{m})}^2 \geq \afcoer \| \bsigma \|_{0,h}^2, ~~ \forall \bsigma \in \HHJ_{h}^{m}, ~ \forall h > 0,
\end{split}
\end{equation}
where $\afcoer$ depends on $\binvelasten$.
\end{lemma}
\begin{proof}
Clearly, $\afman{m}{\bsigma}{\bsigma} \geq C_0 \| \bsigma \|_{L^2(\Dom^{m})}^2$, where $C_0$ depends on $\invelasten{\gamma}{\omega}{\alpha}{\beta}$.
Furthermore, by \cref{eqn:HHJ_tensor_norm_equiv}, $\| \bsigma \|_{L^2(\Dom^{m})} \geq C^{-1} \| \bsigma \|_{0,h}$, so then $\afcoer := C_0 / C^2$.
\end{proof}

%------------------------------------------------------------------------
\subsubsection{Inf-Sup}%\label{sec:}
%------------------------------------------------------------------------

Next, we have a curved element version of the inf-sup condition in \cite[Lem. 5.1]{Blum_CM1990}.
\begin{lemma}\label{lem:mesh_depend_b_inf_sup_flat_domain}
Assume the domain $\Dom^{m} \subset \R^2$ is piecewise smooth consisting of curved elements as described in \cref{sec:curved_FEM}.
Then there is a constant $\bfcoer > 0$, independent of $h$ and $m$, such that for all $h$ sufficiently small
\begin{equation}\label{eqn:mesh_depend_b_inf_sup_flat_domain}
\begin{split}
\sup_{\bvarphi \in \V_h^{m}} \frac{|\bfman{m}{\bvarphi}{v}|}{\| \bvarphi \|_{0,h,m}}
\geq \bfcoer \| v \|_{2,h,m}, \quad \forall v \in \W_h^{m}, ~ \forall h > 0.
\end{split}
\end{equation}
\end{lemma}
\begin{proof}
We start with the case $m=1$ in \cite[Lem. 5.1]{Blum_CM1990}:
\begin{equation}\label{eqn:mesh_depend_b_inf_sup_flat_domain_pf_1}
\begin{split}
\sup_{\hat{\bvarphi} \in \V_h^{1}} \frac{|\bfman{1}{\hat{\bvarphi}}{\hat{v}}|}{\| \hat{\bvarphi} \|_{0,h,1}}
\geq C_0 \| \hat{v} \|_{2,h,1}, \quad \forall \hat{v} \in \W_h^{1}, ~ \forall h > 0,
\end{split}
\end{equation}
on the piecewise linear domain $\Dom^{1}$ with triangulation $\TkDom{h}{1}$, and holds for any degree $r \geq 0$ of the HHJ space.

Consider the map $\bPhi_{T} : T^{1} \to T^{m}$ from \cref{thm:mapping_forms} with approximation properties given by \cref{eqn:curved_elem_map_optim}.
Let $v \in \W_h^{m}$, so by definition, there exists $\hat{v} \in \W_h^{1}$ such that $v \circ \bPhi_{T} = \hat{v}$ on each $T^{1} \in \TkDom{h}{1}$.  By \cref{eqn:equiv_norms_zero_bdy}, $\| v \|_{2,h,m} \approx \| \hat{v} \|_{2,h,1}$.  By \cref{eqn:matrix_Piola_transf_contra} (using $\bPhi_{T}$), for any $\bvarphi \in \V_h^{m}$, there exists $\hat{\bvarphi} \in \V_h^{1}$, such that $\| \bvarphi \|_{0,h,m} \approx \| \hat{\bvarphi} \|_{0,h,1}$, by \cref{eqn:HHJ_tensor_norm_equiv}.

We will use \cref{eqn:mapping_b_h_form} to estimate $\left| \bfman{m}{\bvarphi}{v} \right|$ when $l=1$.  Upon recalling the norm \cref{eqn:scalar_H2_h_norm}, because of boundary conditions, we have that the last term in \cref{eqn:mapping_b_h_form} bounds as
\begin{equation}\label{eqn:mesh_depend_b_inf_sup_flat_domain_pf_2}
\begin{split}
	O(h) \sum_{E^{1} \in \EkBdyDom{h}{1}} \| \hatvarphinn \|_{L^2(E^{1})} \| \hat{\bcn} \cdot \grad{} \Ilag_{h}^{1,1} \hat{v} \|_{L^2(E^{1})} \leq C h \| \hat{\bvarphi} \|_{0,h,1} \| \hat{v} \|_{2,h,1}.
\end{split}
\end{equation}
Moreover, applying basic estimates to \cref{eqn:mapping_b_h_form} yields
\begin{equation}\label{eqn:mesh_depend_b_inf_sup_flat_domain_pf_3}
\begin{split}
	\bfman{m}{\bvarphi}{v} \geq \bfman{1}{\hat{\bvarphi}}{\hat{v}} - C h \| \hat{\bvarphi} \|_{0,h,1} \| \hat{v} \|_{2,h,1} - C \| \hat{\bvarphi} \|_{L^2(\strip^{1})} \| \grad{} \hat{v} \|_{L^2(\strip^{1})}.
\end{split}
\end{equation}
Since $\| \grad{} \hat{v} \|_{L^2(\strip^{1})} \leq \| \grad{} \hat{v} \|_{L^\infty(\Dom^{1})} \| 1 \|_{L^2(\strip^{1})} \leq  \| \grad{} \hat{v} \|_{L^\infty(\Dom^{1})} |\BdyDom|^{1/2} h^{1/2}$, by the quasi-uniform mesh assumption, and $\| \grad{} \hat{v} \|_{L^{\infty}(\Dom^{1})}^2 \leq C (1 + \ln h) (\| \hat{v} \|_{2,h,1}^2 + \| \grad{} \hat{v} \|_{L^2(\Dom^{1})}^2)$ (see \cite[eqn. (4)]{Brenner2004}), we get
\begin{equation}\label{eqn:mesh_depend_b_inf_sup_flat_domain_pf_4}
\begin{split}
	\bfman{m}{\bvarphi}{v} \geq \bfman{1}{\hat{\bvarphi}}{\hat{v}} - C h^{1/2 - \epsilon} \| \hat{\bvarphi} \|_{0,h,1} \| \hat{v} \|_{2,h,1},
\end{split}
\end{equation}
for some small $\epsilon > 0$.  Dividing by $\| \bvarphi \|_{0,h,m}$ and using equivalence of norms, we get
\begin{equation}\label{eqn:mesh_depend_b_inf_sup_flat_domain_pf_5}
\begin{split}
	\frac{\bfman{m}{\bvarphi}{v}}{ \| \bvarphi \|_{0,h,m}} \geq \frac{\bfman{1}{\hat{\bvarphi}}{v}}{ \| \hat{\bvarphi} \|_{0,h,1}} - C h^{1/2 - \epsilon} \| \hat{v} \|_{2,h,1}.
\end{split}
\end{equation}
Taking the supremum, using \cref{eqn:mesh_depend_b_inf_sup_flat_domain_pf_1}, and equivalence of norms, proves \cref{eqn:mesh_depend_b_inf_sup_flat_domain} when $h$ is sufficiently small.
\end{proof}

\begin{remark}\label{rem:inf_sup_W_norm}
By \cref{eqn:scalar_H2_h_norm_trace_Poincare}, \cref{eqn:mesh_depend_b_inf_sup_flat_domain} holds with $\| v \|_{2,h,m}$ replaced by $| v |_{H^{1}(\Dom^{m})}$ with a different inf-sup constant.
\end{remark}

Therefore, \cref{eqn:bounded_a_form}, \cref{eqn:bounded_b_form}, \cref{eqn:a_coercive_over_kernel_b_h}, and \cref{eqn:mesh_depend_b_inf_sup_flat_domain} imply by the standard theory of mixed methods that \cref{eqn:Kirchhoff_manifold_HHJ_mixed} is well-posed in the mesh dependent norms.

%------------------------------------------------------------------------
\section{Error Analysis}\label{sec:error_analysis}
%------------------------------------------------------------------------

We now prove convergence of the HHJ method while accounting for the approximation of the domain using the theory of curved elements described in \cref{sec:curved_FEM}.
The main difficulties are dealing with higher derivatives of the nonlinear map and handling the jump terms in the mesh dependent norms when affected by a nonlinear map. The key ingredients here are \cref{thm:curved_elem_map_estimates}, \cref{eqn:b_h_form_interp_oper_Fortin}, and the following crucial choice of optimal map: let $\widetilde{\MapT{T}^{m}} : T^{1} \to T^{m}$, for all $T^{1} \in \TkDom{h}{1}$ and $1 \leq m \leq k$, be given by
\begin{equation}\label{ass:HHJ_Lag_interp_optimal_map}
\begin{split}
	\MapT{T}^{m} \equiv \widetilde{\MapT{T}^{m}} := \Ilag_{h}^{1,m} \MapT{T}^{} \equiv \Ilag_{h}^{1,m} \bPsi_{T}^{1},
\end{split}
\end{equation}
where $\Ilag_{h}^{1,m}$ is the Lagrange interpolation operator in \cref{eqn:scalar_space_interp_oper} onto degree $m$ polynomials, and we abuse notation by writing $\MapT{T}^{m} \equiv \widetilde{\MapT{T}^{m}}$.

\begin{remark}\label{rem:choice_optimal_map}
Note that $\widetilde{\MapT{T}^{m}}$ is an optimal map because of the approximation properties of $\Ilag_{h}^{1,m}$; hence, the results of \cref{thm:curved_elem_map_estimates} apply to $\widetilde{\MapT{T}^{m}}$.  This choice is necessary to guarantee optimal convergence of the HHJ method when $m=r+1$.  If $m > r+1$, the standard Lenoir map suffices.
\end{remark}

In deriving the error estimates, we use the following regularity result for the Kirchhoff plate problem (see \cite[Thm. 2]{Blum_MMAS1980}, \cite[Table 1]{Blum_CM1990}).
\begin{theorem}\label{thm:Kirchhoff_regularity}
Assume $\Dom$ satisfies the assumptions in \cref{sec:bdy_assume} and let $f \in H^{-1}(\Dom)$.  Then the weak solution $w \in \SW$ of \cref{eqn:weak_form} always satisfies $w \in W^{3,p}(\Dom)$ for some value of $p \in (p_0,2]$, where $1 \leq p_0 < 2$ depends on the angles at the corners of $\Dom$.
\end{theorem}
For technical reasons, we also assume $p > 3/2$ in \cref{thm:Kirchhoff_regularity} (recall \cref{rem:assume_p_3/2}).
Higher regularity (e.g., $w \in H^3(\Dom)$) is achieved if the corner angles are restricted.  In addition, if $f \in L^2(\Dom)$, then $w \in H^4(\Dom)$.  See \cite{Blum_MMAS1980,Blum_CM1990} for more details.
%\SWW{Is this result really true when $p > 1$ is close to $1$?}

%------------------------------------------------------------------------
\subsection{Estimate the PDE Error}\label{sec:estimate_PDE_error}
%------------------------------------------------------------------------

We start with an error estimate that ignores the geometric error, i.e., the continuous and discrete problems are posed on the exact domain.
\begin{theorem}\label{thm:err_estim_exact_domain}
Adopt the boundary assumptions in \cref{sec:bdy_assume}. Let $\bsigma \in \V^{}$ and $w \in \W^{}$ solve \cref{eqn:Kirchhoff_manifold_skeleton_mixed} on the true domain $\Dom$, and assume $w \in W^{t,p}(\Dom)$, so then $\bsigma \in W^{t-2,p}(\Dom;\symmat)$, $t \geq 3$, $3/2 < p \leq 2$ (recall \cref{thm:Kirchhoff_regularity}).  Furthermore, let $r \geq 0$ be the degree of $\V_{h}$, and let $\bsigma_{h} \in \V_h^{}$, $w_{h} \in \W_h^{}$ be the discrete solution of \cref{eqn:Kirchhoff_manifold_HHJ_mixed} on $\Dom$. Then, we obtain
\begin{equation}\label{eqn:err_estim_exact_domain}
\begin{split}
\| \bsigma - \bsigma_{h} \|_{0,h} + \| \grad{} (w - w_{h}) \|_{L^2(\Dom)} &\leq C h^{\min(r+2,t-1) - 2/p}, \\
\text{ when } r \geq 1: \quad \| w - w_{h} \|_{2,h} &\leq C h^{\min(r+1,t-1)-2/p}, \\
\text{ when } r = 0: \quad \| \grad{} (w - w_{h}) \|_{L^2(\Dom)} &\leq C h,
\end{split}
\end{equation}
where $C > 0$ depends on $f$, the domain $\Dom$, and the shape regularity of the mesh.
\end{theorem}
\begin{proof}
With coercivity and the inf-sup condition in hand, the proof is a standard application of error estimates for mixed methods and is given in section SM6. %\cref{sec:pf_error_estim_exact_domain}.
\end{proof}
Note that the last line of \cref{eqn:err_estim_exact_domain} generalizes \cite[Thm. 5.1]{Blum_CM1990} to curved domains.

%------------------------------------------------------------------------
\subsection{Estimate the Geometric Error}\label{sec:estimate_Geom_error}
%------------------------------------------------------------------------

We now approximate the domain using curved, Lagrange mapped triangle elements.
\begin{lemma}\label{lem:err_estim_exact_discrete}
Recall the map $\bPsi^{m} : \Dom^{m} \to \Dom$, with $\bPsi_{T}^{m} := \bPsi^{m} |_{T}$, from \cref{sec:curved_triangulations}, and adopt \cref{ass:HHJ_Lag_interp_optimal_map}. For convenience, set $\bJ = \grad{} \bPsi^{m}$. 
Adopt the boundary assumptions in \cref{sec:bdy_assume}.  Let $\hat{\bsigma}_{h} \in \V_h^{m}$, $\hat{w}_{h} \in \W_h^{m}$ be the discrete solution of \cref{eqn:Kirchhoff_manifold_HHJ_mixed}, with $f$ replaced by $\tilde{f} := f \circ \bPsi^{m} (\det \bJ)$.  Take $(\bsigma_{h},w_{h})$ from \cref{thm:err_estim_exact_domain}, and let $\widetilde{\bsigma}_{h} \in \V_h^{m}$, $\tilde{w}_{h} \in \W_h^{m}$ be the mapped discrete solutions onto $\Dom^{m}$ using \cref{eqn:matrix_Piola_transf_contra}.  In other words, $\bsigma_{h} \circ \bPsi^{m} = (\det \bJ)^{-2} \bJ \widetilde{\bsigma}_{h} \bJ\tp$ and $\tilde{w}_{h} = w_{h} \circ \bPsi^{m}$, defined element-wise.  Similarly, we map the test functions $\bvarphi_{h} \in \V_h^{}$, $v_{h} \in \W_h^{}$ to $\hat{\bvarphi}_{h} \in \V_h^{m}$, $\hat{v}_{h} \in \W_h^{m}$.
Then, we obtain the error equations for the geometric error:
\begin{equation}\label{eqn:exact_discrete_error_eqn}
\afman{m}{\widetilde{\bsigma}_{h} - \hat{\bsigma}_{h}}{\hat{\bvarphi}_{h}} +
\bfman{m}{\hat{\bvarphi}_{h}}{\tilde{w}_{h} - \hat{w}_{h}} +
\bfman{m}{\widetilde{\bsigma}_{h} - \hat{\bsigma}_{h}}{\hat{v}_{h}} = \mathrm{E}_{0} (\hat{\bvarphi}_{h},\hat{v}_{h}),
\end{equation}
for all $(\hat{v}_{h}, \hat{\bvarphi}_{h}) \in \W^{m}_h \times \V^{m}_h$, where
\begin{equation}\label{eqn:geometric_consistency_error}
| \mathrm{E}_{0} (\hat{\bvarphi}_{h},\hat{v}_{h}) | \leq C h^{q} \left( \| \hat{\bvarphi}_{h} \|_{0,h,m} + \| \hat{v}_{h} \|_{2,h,m} \right) \| f \|_{H^{-1}(\Dom)},
\end{equation}
where $q = m$ when $m=r+1$, otherwise $q = m-1$.
\end{lemma}

\begin{proof}
We will use \cref{eqn:mapping_b_h_form} with $m$, $l$ replaced by $\infty$, $m$, respectively. First, note that $(\ident{\Dom^{m}} - \bPsi^{m}) \cdot \CProj \grad{} \hat{v}_{h} \in H^{2}_{h}(\Dom^{m})$, because $(\ident{T^{m}} - \bPsi^{m}_{T})$ is zero at all internal edges, and is identically zero on all elements not in $\TkBdyDom{h}{m}$. 
Upon noting $v_{h} \in \zerobdy{H}^1 (\Dom)$ and  \cref{eqn:mesh_depend_b_inf_sup_flat_domain_pf_2}, straightforward manipulation gives
\begin{equation}\label{eqn:exact_discrete_error_eqn_pf_1}
\begin{split}
\bfman{}{\bvarphi_{h}}{v_{h}} &= \bfman{m}{\hat{\bvarphi}_{h}}{\hat{v}_{h}} + \bfman{m}{\hat{\bvarphi}_{h}}{(\ident{\Dom^{m}} - \bPsi^{m}) \cdot \CProj \grad{} \hat{v}_{h}} \\
&\qquad + O(h^{m}) \| \hat{\bvarphi}_{h} \|_{0,h,m} \| \hat{v}_{h} \|_{2,h,m}.
\end{split}
\end{equation}
Taking advantage of \cref{eqn:curved_elem_map_estimate_isolate}, we get
\begin{equation}\label{eqn:exact_discrete_error_eqn_pf_2}
\begin{split}
\bfman{}{\bvarphi_{h}}{v_{h}} &= \bfman{m}{\hat{\bvarphi}_{h}}{\hat{v}_{h}} + \bfman{1}{\check{\bvarphi}_{h}}{(\MapT{}^{m} - \MapT{}^{}) \cdot \CProj \grad{} \hat{v}_{h}} \\
&\qquad + O(h^{m}) \| \hat{\bvarphi}_{h} \|_{0,h,m} \| \hat{v}_{h} \|_{2,h,m},
\end{split}
\end{equation}
where $\check{\bvarphi}_{h} \in \V_h^{1}$ and $\MapT{}^{m} := \Ilag_{h}^{1,m} \MapT{}^{}$, by \cref{ass:HHJ_Lag_interp_optimal_map}.  If $m=r+1$, the Fortin property \cref{eqn:b_h_form_interp_oper_Fortin} yields $\bfman{1}{\check{\bvarphi}_{h}}{(\MapT{}^{m} - \MapT{}^{}) \cdot \CProj \grad{} \hat{v}_{h}} = 0$.  If $m \neq r + 1$, then a straightforward estimate shows $\bfman{1}{\check{\bvarphi}_{h}}{(\MapT{}^{m} - \MapT{}^{}) \cdot \CProj \grad{} \hat{v}_{h}} \leq C h^{m-1} \| \hat{\bvarphi}_{h} \|_{0,h,m} \| \hat{v}_{h} \|_{2,h,m}$, where we used equivalence of norms \cref{eqn:equiv_norms_zero_bdy,eqn:HHJ_tensor_norm_equiv}.

Therefore, using \cref{eqn:mapping_a_form} and \cref{eqn:mapping_b_h_form_part_2}, the first line in \cref{eqn:Kirchhoff_manifold_HHJ_mixed} (with $m=\infty$) maps to
\begin{equation}\label{eqn:Kirchhoff_discrete_mixed_mapped_E1}
\begin{split}
\afman{m}{\widetilde{\bsigma}_{h}}{\hat{\bvarphi}_{h}} + \bfman{m}{\hat{\bvarphi}_{h}}{\tilde{w}_{h}} &= \mathrm{I}_{1}, ~ \forall \hat{\bvarphi}_{h} \in \V_h^{m},
\end{split}
\end{equation}
where $1 \leq m \leq k$ and $C > 0$ is a constant depending only on $\Dom$ such that
\begin{equation}\label{eqn:geometric_consistency_terms_1}
\begin{split}
\mathrm{I}_1 &\leq C h^{q} \| \hat{\bvarphi}_{h} \|_{L^2(\strip^{m})} \left( \| \widetilde{\bsigma}_{h} \|_{L^2(\strip^{m})} + \| \tilde{w}_{h} \|_{2,h,m} \right),
\end{split}
\end{equation}
where $q$ was defined earlier. 
The second equation in \cref{eqn:Kirchhoff_manifold_HHJ_mixed} (with $m=\infty$) maps to
\begin{equation}\label{eqn:Kirchhoff_discrete_mixed_mapped_E2}
\begin{split}
\bfman{m}{\widetilde{\bsigma}_{h}}{\hat{v}_h} &= -\duality{f \circ \bPsi^{m} (\det \bJ)}{\hat{v}_h}{\Dom^{m}} + \mathrm{I}_2, \quad \forall \hat{v}_h \in \W_h^{m},
\end{split}
\end{equation}
where, for some constant $C > 0$ depending only on $\Dom$,
\begin{equation}\label{eqn:geometric_consistency_terms_2}
\begin{split}
\mathrm{I}_2 &\leq C h^{q} \| \widetilde{\bsigma}_{h} \|_{L^2(\strip^{m})} \| \hat{v}_{h} \|_{2,h,m}.
\end{split}
\end{equation}
Then, subtracting \cref{eqn:Kirchhoff_manifold_HHJ_mixed} (with $1 \leq m \leq k$) for the solution $(\hat{\bsigma}_{h},\hat{w}_{h})$ from the above equations, combining everything, and noting the a priori estimate \cref{eqn:Kirchhoff_manifold_HHJ_mixed_a_priori_est} gives \cref{eqn:exact_discrete_error_eqn,eqn:geometric_consistency_error}.
\end{proof}

\begin{theorem}\label{thm:err_estim_exact_discrete}
Adopt the hypothesis of \cref{lem:err_estim_exact_discrete}.  Then, the following error estimate holds
\begin{equation}\label{eqn:err_estim_exact_discrete}
\begin{split}
	\| \widetilde{\bsigma}_{h} - \hat{\bsigma}_{h} \|_{0,h,m} + \| \tilde{w}_{h} - \hat{w}_{h} \|_{2,h,m} &\leq C h^{q} \| f \|_{H^{-1}(\Dom)},
\end{split}
\end{equation}
for some uniform constant $C > 0$.
\end{theorem}
\begin{proof}
From \cref{eqn:exact_discrete_error_eqn}, choose $\hat{v}_{h} = 0$ and use \cref{lem:mesh_depend_b_inf_sup_flat_domain} to get
\begin{equation}\label{eqn:err_estim_exact_discrete_pf_1}
\begin{split}
	\bfcoer \| \tilde{w}_{h} - \hat{w}_{h} \|_{2,h,m} &\leq \sup_{\hat{\bvarphi}_{h} \in \V_h^{m}} \frac{|\bfman{m}{\hat{\bvarphi}_{h}}{\tilde{w}_{h} - \hat{w}_{h}}|}{\| \hat{\bvarphi}_{h} \|_{0,h,m}} \\
	&\leq \sup_{\hat{\bvarphi}_{h} \in \V_h^{m}} \frac{|\afman{m}{\widetilde{\bsigma}_{h} - \hat{\bsigma}_{h}}{\hat{\bvarphi}_{h}}| + |\mathrm{E}_{0} (\hat{\bvarphi}_{h},0)|}{\| \hat{\bvarphi}_{h} \|_{0,h,m}} \\
	&\leq C \afcont \| \widetilde{\bsigma}_{h} - \hat{\bsigma}_{h} \|_{0,h,m} + C h^{q} \| f \|_{H^{-1}(\Dom)},
\end{split}
\end{equation}
where we used the norm equivalence \cref{eqn:HHJ_tensor_norm_equiv}.  Next, choose $\hat{\bvarphi}_{h} = \widetilde{\bsigma}_{h} - \hat{\bsigma}_{h}$ and $\hat{v}_{h} = -(\tilde{w}_{h} - \hat{w}_{h})$ in \cref{eqn:exact_discrete_error_eqn} to get
\begin{equation}\label{eqn:err_estim_exact_discrete_pf_2}
\begin{split}
	\afcoer \| \widetilde{\bsigma}_{h} - & \hat{\bsigma}_{h} \|_{L^2(\Dom^{m})}^2 \leq \afman{m}{\widetilde{\bsigma}_{h} - \hat{\bsigma}_{h}}{\widetilde{\bsigma}_{h} - \hat{\bsigma}_{h}} \\
	&\leq C h^{q} \left( \| \widetilde{\bsigma}_{h} - \hat{\bsigma}_{h} \|_{0,h,m} + \| \tilde{w}_{h} - \hat{w}_{h} \|_{2,h,m} \right) \| f \|_{H^{-1}(\Dom)} \\
	&\leq C h^{q} \left( \| \widetilde{\bsigma}_{h} - \hat{\bsigma}_{h} \|_{0,h,m} + C h^{q} \| f \|_{H^{-1}(\Dom)} \right) \| f \|_{H^{-1}(\Dom)} \\
	&\leq C (h^{q})^2 \| f \|_{H^{-1}(\Dom)}^2 + C h^{q} \| \widetilde{\bsigma}_{h} - \hat{\bsigma}_{h} \|_{L^2(\Dom^{m})} \| f \|_{H^{-1}(\Dom)} \\
	&\leq C (h^{q})^2 \| f \|_{H^{-1}(\Dom)}^2 + \frac{\afcoer}{2} \| \widetilde{\bsigma}_{h} - \hat{\bsigma}_{h} \|_{L^2(\Dom^{m})}^2,
\end{split}
\end{equation}
where we used \cref{eqn:err_estim_exact_discrete_pf_1}, norm equivalence \cref{eqn:HHJ_tensor_norm_equiv}, and a weighted Cauchy inequality.  Then, by combining the above results, we get the assertion.
\end{proof}

%------------------------------------------------------------------------
\subsection{Estimate the Total Error}\label{sec:estimate_total_error}
%------------------------------------------------------------------------

We will combine \cref{thm:err_estim_exact_domain} and \cref{thm:err_estim_exact_discrete} to get the total error.
\begin{theorem}[general error estimate]\label{thm:general_error_estim}
Adopt the hypotheses of \cref{thm:err_estim_exact_domain,lem:err_estim_exact_discrete}.  If $m \geq r + 1$, then
\begin{equation}\label{eqn:general_error_estim}
\begin{split}
\| \bsigma - \hat{\bsigma}_{h} \circ (\bPsi^{m})^{-1} \|_{0,h} + \| \grad{} (w - \hat{w}_{h} \circ (\bPsi^{m})^{-1}) \|_{L^2(\Dom)} &\leq C h^{\min(r+2,t-1) - 2/p}, \\
r \geq 1: ~~ \| w - \hat{w}_{h} \circ (\bPsi^{m})^{-1} \|_{2,h} &\leq C h^{\min(r+1,t-1)-2/p}, \\
r = 0: ~~ \| \grad{} (w - \hat{w}_{h} \circ (\bPsi^{m})^{-1}) \|_{L^2(\Dom)} &\leq C h,
\end{split}
\end{equation}
where $C > 0$ depends on $f$, the domain $\Dom$, and the shape regularity of the mesh.
\end{theorem}
\begin{proof}
By the triangle inequality and using the properties of the map $\bPsi^{m}$, we have
\begin{equation}\label{eqn:general_error_estim_pf_1}
\begin{split}
	\| \bsigma - \hat{\bsigma}_{h} \circ (\bPsi^{m})^{-1} \|_{0,h} &\leq \| \bsigma - \widetilde{\bsigma}_{h} \circ (\bPsi^{m})^{-1} \|_{0,h} \\
	&\quad + \| \widetilde{\bsigma}_{h} \circ (\bPsi^{m})^{-1} - \hat{\bsigma}_{h} \circ (\bPsi^{m})^{-1} \|_{0,h} \\
	\leq \| \bsigma - \bsigma_{h} \|_{0,h} + & \| \bsigma_{h} - \widetilde{\bsigma}_{h} \circ (\bPsi^{m})^{-1} \|_{0,h} + C \| \widetilde{\bsigma}_{h} - \hat{\bsigma}_{h} \|_{0,h,m}.
\end{split}
\end{equation}
Focusing on the middle term, we have
\begin{equation}\label{eqn:general_error_estim_pf_2}
\begin{split}
	\| & \bsigma_{h} - \widetilde{\bsigma}_{h} \circ (\bPsi^{m})^{-1} \|_{0,h} \leq \| \bsigma_{h} \circ \bPsi^{m} - \widetilde{\bsigma}_{h} \|_{0,h,m} \\
	&\leq \| (\det \bJ)^{-2} \bJ \widetilde{\bsigma}_{h} \bJ\tp - \widetilde{\bsigma}_{h} \|_{0,h,m} \leq C h^{r+1} \| \widetilde{\bsigma}_{h} \|_{0,h,m} \leq C h^{r+1} \| f \|_{H^{-1}(\Dom)}.
\end{split}
\end{equation}
Combining everything, we get
\begin{equation}\label{eqn:general_error_estim_pf_3}
\begin{split}
	\| \bsigma - \hat{\bsigma}_{h} \circ (\bPsi^{m})^{-1} \|_{0,h} &\leq C \max \left( h^{r+1}, h^{\min(r+2,t-1) - 2/p} \right),
\end{split}
\end{equation}
where $C > 0$ depends on $f$.  Taking a similar approach for the other terms involving $w - \hat{w}_{h} \circ (\bPsi^{m})^{-1}$ delivers the estimates.
\end{proof}

\begin{corollary}\label{cor:general_error_estim_smooth}
Adopt the hypothesis of \cref{thm:general_error_estim}, but assume $\BdyDom$ is globally smooth, and $f$, $w$, and $\bsigma$ are smooth.  If $r \geq 0$ is the degree of $\V_{h}$, then
\begin{equation}\label{eqn:general_error_estim_smooth}
\begin{split}
	\| \bsigma - \hat{\bsigma}_{h} \circ (\bPsi^{m})^{-1} \|_{0,h} + \| \grad{} (w - \hat{w}_{h} \circ  (\bPsi^{m})^{-1}) \|_{L^2(\Dom)} & \\
	+ h \| w - \hat{w}_{h} \circ (\bPsi^{m})^{-1} & \|_{2,h} \leq C h^{r+1},
\end{split}
\end{equation}
where $C > 0$ depends on $w$, the domain $\Dom$, and the shape regularity of the mesh.
\end{corollary}

\begin{remark}\label{rem:sub_parametric}
From \cref{thm:err_estim_exact_discrete}, if $m < r+1$, the error is sub-optimal, i.e., is $O(h^{m-1})$ for a smooth solution.  However, this only occurs in $\TkBdyDom{h}{}$; in the rest of the mesh, it is $O(h^{r+1})$.  Since the mesh is quasi-uniform, a straightforward estimate gives that the error measured over the entire domain is $O(h^{m-1/2})$.  This is verified in the simply supported numerical example in \cref{sec:simsupp_disk_ex}, as well as in both examples in \cref{sec:clamped_three_leaf_ex}, \cref{sec:simsupp_three_leaf_ex}.
\end{remark}

%------------------------------------------------------------------------
\subsection{Inhomogeneous Boundary Conditions}\label{sec:nonzero_BCs}
%------------------------------------------------------------------------

We now explain how to extend the above theory to handle non-vanishing boundary conditions.  First, construct a function $g \in W^{t,p}(\Dom)$, such that the displacement satisfies $w=g$ on $\BdyDom$ and $\pd{\bcn} w = \pd{\bcn} g$ on $\Bclamped$, and construct a function $\brho \in W^{t-2,p}(\Dom;\symmat)$, such that the normal-normal moment satisfies $\bcn\tp \bsigma \bcn = \bcn\tp \brho \bcn$ on $\Bsimsupp$, where $t \geq 3$, $3/2 < p \leq 2$ (recall \cref{thm:Kirchhoff_regularity}).

Then \cref{eqn:Kirchhoff_manifold_skeleton_mixed} is replaced by the problem of determining $(\bsigma,w) = (\zerobdy{\bsigma} + \brho,\zerobdy{w} + g)$, with $\zerobdy{\bsigma} \in \SV_{h}$, $\zerobdy{w} \in \SW_{h}$ (i.e., with homogeneous boundary conditions) such that
\begin{equation}\label{eqn:Kirchhoff_skeleton_mixed_nonzero_BC}
\begin{split}
	\afman{}{\zerobdy{\bsigma}}{\bvarphi} + \bfman{}{\bvarphi}{\zerobdy{w}} &= - \afman{}{\brho}{\bvarphi} - \bfman{}{\bvarphi}{g} + \inner{\varphinn}{\bcn \cdot \grad{} g}{\Bclamped}, \quad \forall \bvarphi \in \SV_{h}, \\
	\bfman{}{\zerobdy{\bsigma}}{v} &= -\duality{f}{v}{\Dom} - \bfman{}{\brho}{v}, ~ \forall v \in \SW_{h}.
\end{split}
\end{equation}
Note that the right-hand-side in the first equation of \cref{eqn:Kirchhoff_skeleton_mixed_nonzero_BC} simplifies to $- \afman{}{\brho}{\bvarphi} - \ringbfman{}{\bvarphi}{g}$, where $\ringbfman{}{\bvarphi}{v} := \bfman{}{\bvarphi}{v} - \inner{\varphinn}{\bcn \cdot \grad{} v}{\Bclamped}$ (i.e., it has no boundary term).

Similarly, the corresponding (intermediate) discrete problem \cref{eqn:Kirchhoff_manifold_HHJ_mixed}, on the exact domain, is replaced by finding $(\bsigma_{h},w_{h}) = (\zerobdy{\bsigma}_{h} + \brho_{h}, \zerobdy{w}_{h} + g_{h})$, with $\zerobdy{\bsigma}_{h} \in \V_h^{}$, $\zerobdy{w}_{h} \in \W_h^{}$ such that
\begin{equation}\label{eqn:Kirchhoff_HHJ_mixed_nonzero_BC_exact_Dom}
\begin{split}
\afman{}{\zerobdy{\bsigma}_{h}}{\bvarphi_{h}} + \bfman{}{\bvarphi_{h}}{\zerobdy{w}_{h}} &= - \afman{}{\brho_{h}}{\bvarphi_{h}} - \ringbfman{}{\bvarphi_{h}}{g_{h}} \\
& - \inner{\varphinn_{h}}{\bcn \cdot \grad{} g_{h}}{\Bclamped} + \inner{\varphinn_{h}}{\bcn \cdot \grad{} g}{\Bclamped}, ~ \forall \bvarphi_{h} \in \V_h^{}, \\
\bfman{}{\zerobdy{\bsigma}_{h}}{v_{h}} &= -\duality{f}{v_{h}}{\Dom} - \bfman{}{\brho_{h}}{v_{h}}, ~ \forall v_{h} \in \W_h^{},
\end{split}
\end{equation}
where $\brho_{h} = \LtwoProj^{} \brho$, and
$\LtwoProj : H^{0}_{h} (\Dom) \to \V_h$ is the $L^2(\Dom)$ projection, i.e., $\brho_{h}$ satisfies
\begin{equation}\label{eqn:mesh_dep_skeleton_L2_proj}
\begin{split}
\inner{\brho_{h} -\brho}{\bvarphi_{h}}{\TkDom{h}{}} + \duality{\bcn\tp [\brho_{h} - \brho] \bcn}{\varphinn_{h}}{\EkDom{h}{}} = 0, \text{ for all } \bvarphi_{h} \in \V_h^{},
\end{split}
\end{equation}
and $g_{h} = \Ilag_{h} g$. An error estimate between the solutions of \cref{eqn:Kirchhoff_skeleton_mixed_nonzero_BC} and \cref{eqn:Kirchhoff_HHJ_mixed_nonzero_BC_exact_Dom}, analogous to \cref{thm:err_estim_exact_domain}, follows similarly with the following additional steps. First, estimate $\bfman{}{\brho - \brho_{h}}{v_{h}} \leq \| \brho - \brho_{h} \|_{0,h} \| v_{h} \|_{2,h}$, note $\| \brho - \brho_{h} \|_{0,h} \leq \| \brho - \IHHJ_{h} \brho \|_{0,h}$ and use the approximation properties of $\IHHJ_{h}$ in subsection SM4.3.%\cref{sec:HHJ_approx_results}.
Next, estimate $\ringbfman{}{\bvarphi_{h}}{g - g_{h}}$ and $\inner{\varphinn_{h}}{\bcn \cdot \grad{} (g - g_{h})}{\Bclamped}$ with \cref{eqn:b_h_form_curved_Fortin_oper}.

Finally, the discrete problem on the discrete domain is to find $(\hat{\bsigma}_{h},\hat{w}_{h}) = (\zerobdy{\hat{\bsigma}}_{h} + \hat{\brho}_{h},\zerobdy{\hat{w}}_{h} + \hat{g}_{h})$, with $\zerobdy{\hat{\bsigma}}_{h} \in \V_h^{m}$, $\zerobdy{\hat{w}}_{h} \in \W_h^{m}$ such that
\begin{equation}\label{eqn:Kirchhoff_HHJ_mixed_nonzero_BC}
\begin{split}
	\afman{m}{\zerobdy{\hat{\bsigma}}_{h}}{\hat{\bvarphi}_{h}} + \bfman{m}{\hat{\bvarphi}_{h}}{\zerobdy{\hat{w}}_{h}} &= - \afman{m}{\hat{\brho}_{h}}{\hat{\bvarphi}_{h}} - \ringbfman{m}{\hat{\bvarphi}_{h}}{\hat{g}_{h}} \\ 
	& - \inner{\hatvarphinn_{h}}{\hat{\bcn} \cdot [\grad{} \hat{g}_{h} - \tilde{\bxi}]}{\Bclamped^{m}},
	~ \forall \hat{\bvarphi}_{h} \in \V_h^{m}, \\
	\bfman{m}{\zerobdy{\hat{\bsigma}}_{h}}{\hat{v}_{h}} &= -\duality{\tilde{f}}{\hat{v}_{h}}{\Dom^{m}} - \bfman{m}{\hat{\brho}_{h}}{\hat{v}_{h}}, ~ \forall \hat{v}_{h} \in \W_h^{m},
\end{split}
\end{equation}
where $\hat{\brho}_{h} := \LtwoProj^{m} \widetilde{\brho}$, with $\widetilde{\brho}$ given by $\brho \circ \bPsi^{m} |_{T} = (\det \bJ_{T})^{-2} \bJ_{T} \widetilde{\brho} \bJ_{T}\tp$, $\bJ_{T} = \grad{} \bPsi^{m} |_{T}$, and $\LtwoProj^{m} : H^{0}_{h} (\Dom^{m}) \to \V_h^{m}$ is the $L^2(\Dom^{m})$ projection on $\Dom^{m}$, $\hat{g}_{h} := \Ilag_{h}^{m} \tilde{g}$, with $\tilde{g} := g \circ \bPsi^{m}$, and $\tilde{\bxi} = (\grad{} g) \circ \bPsi^{m}$. 
To obtain an analogous result to \cref{thm:general_error_estim}, we need to generalize \cref{lem:err_estim_exact_discrete}.  The argument is mostly the same as the proof of \cref{lem:err_estim_exact_discrete}, except there is an additional step to show that
\begin{equation}\label{eqn:exact_discrete_error_eqn_inhomog_BC}
\afman{m}{\widetilde{\brho} - \hat{\brho}_{h}}{\hat{\bvarphi}_{h}} +
\ringbfman{m}{\hat{\bvarphi}_{h}}{\tilde{g} - \hat{g}_{h}} +
\bfman{m}{\widetilde{\brho} - \hat{\brho}_{h}}{\hat{v}_{h}} = \mathrm{E}_{1} (\hat{\bvarphi}_{h},\hat{v}_{h}),
\end{equation}
for all $(\hat{v}_{h}, \hat{\bvarphi}_{h}) \in \W^{m}_h \times \V^{m}_h$, where
\begin{equation}\label{eqn:geometric_consistency_error_inhomog_BC}
| \mathrm{E}_{1} (\hat{\bvarphi}_{h},\hat{v}_{h}) | \leq C h^{q} \left( \| \hat{\bvarphi}_{h} \|_{0,h,m} + \| \hat{v}_{h} \|_{2,h,m} \right) \left( \| \brho \|_{W^{1,p}(\Dom;\symmat)} + \| g \|_{W^{3,p}(\Dom)} \right),
\end{equation}
where $q = m$ when $m=r+1$, otherwise $q = m-1$.  This also follows the same outline, but we note the following.  
(1) Estimating $\ringbfman{m}{\hat{\bvarphi}_{h}}{\tilde{g} - \hat{g}_{h}}$ with \cref{eqn:mapping_b_h_form} is simpler because the last boundary term in \cref{eqn:mapping_b_h_form} does not appear; then use \cref{lem:b_h_form_curved_Fortin_oper}; (2) noting that $\hat{g}_{h} = g_{h} \circ \bPsi^{m}$, $\inner{\varphinn_{h}}{\bcn \cdot \grad{} (g_{h} - g)}{\Bclamped}$ is mapped to $\inner{\hatvarphinn_{h}}{\hat{\bcn} \cdot \grad{} (\hat{g}_{h} - \tilde{g})}{\Bclamped^{m}}$ (plus residual terms) and is compared against $\inner{\hatvarphinn_{h}}{\hat{\bcn} \cdot [\grad{} \hat{g}_{h} - \tilde{\bxi}]}{\Bclamped^{m}}$; (3) finally, estimate $\inner{\hatvarphinn_{h}}{\hat{\bcn} \cdot [\grad{} \tilde{g} - \tilde{\bxi}]}{\Bclamped^{m}}$ using similar arguments as in the proof of \cref{thm:mapping_forms}. 
With this, and the obvious generalization of \cref{thm:err_estim_exact_discrete}, we obtain the following.
\begin{theorem}[inhomogeneous boundary conditions]\label{thm:error_estim_inhomog_BCs}
Adopt the hypotheses of \cref{thm:general_error_estim}, except assume that $(\bsigma, w)$ solves \cref{eqn:Kirchhoff_skeleton_mixed_nonzero_BC} and $(\hat{\bsigma}_{h},\hat{w}_{h})$ solves \cref{eqn:Kirchhoff_HHJ_mixed_nonzero_BC}.  If $m \geq r + 1$, then $(\bsigma, w)$ and $(\hat{\bsigma}_{h},\hat{w}_{h})$ satisfy the same estimates as in \cref{eqn:general_error_estim}.
\end{theorem}
We also obtain a corollary directly analogous to \cref{cor:general_error_estim_smooth}.

%------------------------------------------------------------------------
\section{Numerical Results}\label{sec:numerical_results}
%------------------------------------------------------------------------

We present numerical examples  computed on a disk, as well as on a non-symmetric domain.  The discrete domains were generated by a successive uniform refinement scheme, with curved elements generated using a variant of the procedure in \cite[Sec.~3.2]{Lenoir_SJNA1986}.  As above, the finite element spaces $\V_{h}$ and $\W_{h}$ are of
degree $r$ and $r+1$ respectively, where $r\ge 0$, and the geometric approximation degree is denoted $m$.  All computations were done with the Matlab/C++ finite element toolbox
FELICITY \cite{Walker_SJSC2018}, where we used the ``backslash'' command in Matlab to solve the linear systems.

From \cref{ass:HHJ_Lag_interp_optimal_map}, recall that $\MapT{}^{m} := \Ilag_{h}^{1,m} \bPsi_{}^{1}$, which is plausible to implement, but inconvenient.  Instead, we first compute $\MapT{}^{m+1}$ using the procedure in \cite[Sec. 3.2]{Lenoir_SJNA1986}, then we define $\MapT{}^{m} := \Ilag_{h}^{1,m} \MapT{}^{m+1}$, which is easy to implement because they are standard Lagrange spaces.  Moreover, the accuracy is not affected.  As for the boundary data, $\hat{g}_{h}$, $\tilde{\bxi}$, and $\hat{\brho}_{h}$ only need to be computed on the boundary $\BdyDom^{m}$; in fact, only the boundary part of the $L^2$ projection $\LtwoProj^{m}$ needs to be computed.

%------------------------------------------------------------------------
\subsection{Practical Error Estimates}\label{sec:practical_error_est}
%------------------------------------------------------------------------

For convenience, the errors we compute are $\| \grad{} (w - \hat{w}_{h}) \|_{L^2(\Dom^{m})}$, $| w - \hat{w}_{h} |_{H^2(\TkDom{h}{m})}$, $\| \bsigma - \hat{\bsigma}_{h} \|_{L^2(\Dom^{m})}$, and $h^{1/2} \| \bcn_{h}\tp (\bsigma - \hat{\bsigma}_{h}) \bcn_{h} \|_{L^2(\EkDom{h}{m})}$, where the exact solution has been extended by analytic continuation.  These errors can be related to the ones in \cref{eqn:general_error_estim_smooth} by basic arguments and a triangle inequality.  Essentially, we need to bound the error between $(\bsigma_{h},w_{h})$ and $(\widetilde{\bsigma}_{h},\tilde{w}_{h})$ in a sense clarified by the following result.

\begin{proposition}\label{prop:error_due_to_mapping}
Let $f_h$ be a piecewise (possibly mapped) polynomial defined over $\TkDom{h}{}$, and let $\bPsi^{m} : \Dom^{m} \to \Dom$ be the piecewise element mapping from \cref{sec:curved_triangulations}.  Assume $f_{h}$ has a bounded extension to $\Dom^{m}$.  Then,
\begin{equation}\label{eqn:error_due_to_mapping}
\begin{split}
\| \grad{}^{s} (f_h - f_h \circ \bPsi^{m}) \|_{L^2(\Dom^{m})} &\leq C h^{m} \| \grad{}^{s} f_h \|_{L^2(\widetilde{\Dom}^{m})}, \text{ for } s = 0, 1, \\
\| f_h - f_h \circ \bPsi^{m} \|_{H^2_{h}(\Dom^{m})} &\leq C h^{m-1} \| f_h \|_{H^2_{h}(\widetilde{\Dom}^{m})},
\end{split}
\end{equation}
where $\widetilde{\Dom}^{m} = \bigcup_{T^{m} \in \TkDom{h}{m}} \mathrm{conv} (T^{m})$ and $\mathrm{conv} (T^{m})$ is the convex hull of $T^{m}$.
\end{proposition}

%------------------------------------------------------------------------
\subsection{Unit Disk Domain}\label{sec:disk_ex}
%------------------------------------------------------------------------

The disk has an unexpected symmetry with respect to curved elements. When approximating $\BdyDom$ by polynomials of degree $m = 2 q$, the approximation order is actually $m = 2 q + 1$.  This is because each circular arc of $\BdyDom$, when viewed as a graph over a flat edge in $\EkBdyDom{h}{1}$, is symmetric about the midpoint of the edge.  Thus, since the Lagrange interpolation nodes are placed symmetrically on the edge, the resulting interpolant must be of even degree. In other words, $\Dom^{2}$, $\Dom^{4}$ have the same approximation order as $\Dom^{3}$, $\Dom^{5}$, respectively.  Our numerical results reflect this.

%------------------------------------------------------------------------
\subsubsection{The Homogeneous Clamped Disk}\label{sec:clamped_disk_ex}
%------------------------------------------------------------------------

In this example, the exact solution with clamped boundary conditions on $\BdyDom$, written in polar coordinates, is taken to be
\begin{equation}\label{eqn:clamped_exact_soln_disk}
	w(r,\theta) = \sin^2 (\pi r).
\end{equation}
\Cref{tab:clamped_disk_EOC} shows the estimated orders of convergence (EoC), which were computed by evaluating the ratio of the errors between the last two meshes in a sequence of successively, uniformly refined meshes.  The optimal orders of convergence, based on the degree of the elements, is $r+1$ for the three quantities $| \hat{w}_{h} |_{H^1}$, $| \hat{\bsigma}_{h} |_{L^2}$, and $h^{1/2} | \hatsignn_{h} |_{L^2(\EkDom{h}{m})}$, and $r$ for $| \hat{w}_{h} |_{H^2_h}$; note: we use the abbreviation $| \hat{w}_{h} |_{H^1} \equiv | w - \hat{w}_{h} |_{H^1(\Dom^{m})}$, etc.  The convergence is better than expected, in that we do not see reduced order convergence when $m=r$, possibly due to the clamped boundary conditions and/or the choice of exact solution (also recall the symmetry discussion earlier).  We do see reduced convergence when $m = r-1$.

\begin{table}[htbp]
{\footnotesize
\caption{EoC for the clamped disk.  $N_T$ is the number of triangles in the final mesh after multiple uniform refinements.  Asterisks indicate suboptimal orders and italics indicate the case $m=r+1$, for which optimality is proven in this paper.}\label{tab:clamped_disk_EOC}
\begin{center}
\begin{tabular}{cccK{.9in}K{.9in}K{.9in}K{.9in}} \toprule
$N_{T}$ & $m$ & $r$ & $| \hat{w}_{h} |_{H^1}$ & $| \hat{w}_{h} |_{H^2_h}$ & $| \hat{\bsigma}_{h} |_{L^2}$ & $h^{1/2} | \hatsignn_{h} |_{L^2(\EkDom{h}{m})}$ \\ \midrule
% num_elem & m & r & H^1 w & H^2 w & L^2 \sig & skel \sig
$2^{17}$ & 1 & 0 & \it 1.0002 & \it 0.0000 & \it 0.9997 & \it 1.0007 \\ 
$2^{15}$ & 1 & 1 & 2.0006 & 0.9998 & 1.9978 & 2.0312 \\ 
$2^{15}$ & 1 & 2 & 2.0052\rlap{*} & 1.9980\rlap{*} & 1.5121\rlap{*} & 1.5022\rlap{*} \\ 
$2^{15}$ & 2 & 1 & \it 2.0002 & \it 0.9990 & \it 1.9976 & \it 2.0317 \\ 
$2^{15}$ & 2 & 2 & 2.9984 & 1.9985 & 2.9994 & 2.9934 \\ 
$2^{13}$ & 2 & 3 & 4.0039 & 3.0007 & 3.9907 & 4.0853 \\ 
$2^{15}$ & 3 & 2 & \it 2.9984 & \it 1.9985 & \it 2.9994 & \it 2.9934 \\ 
$2^{13}$ & 3 & 3 & 4.0039 & 3.0007 & 3.9906 & 4.0746 \\ 
$2^{13}$ & 3 & 4 & 4.9862 & 3.9881 & 3.8387\rlap{*} & 3.5173\rlap{*} \\ 
$2^{13}$ & 4 & 3 & \it 4.0038 & \it 3.0006 & \it 3.9908 & \it 4.0975 \\ 
$2^{13}$ & 4 & 4 & 4.9868 & 3.9883 & 5.0022 & 4.9824 \\ 
$2^{13}$ & 5 & 4 & \it 4.9868 & \it 3.9883 & \it 5.0022 & \it 4.9823 \\
\bottomrule
\end{tabular}
\end{center}
}
\end{table}

%------------------------------------------------------------------------
\subsubsection{The Homogeneous Simply Supported Disk}\label{sec:simsupp_disk_ex}
%------------------------------------------------------------------------

The exact solution with simply supported boundary conditions on $\BdyDom$, written in polar coordinates, is
\begin{equation}\label{eqn:simpsupp_exact_soln_disk}
w(r,\theta) = \cos ((3/2) \pi r).
\end{equation}
\Cref{tab:simpsupp_disk_EOC} shows the estimated orders of convergence (EoC).  The convergence order is consistent with the error estimate in \cref{eqn:general_error_estim_smooth} (accounting for the symmetry of the disk).  For example, when $m=r=1$, we see $O(h^{1/2})$ for $\hat{\bsigma}_{h}$ (see \cref{rem:sub_parametric}). The convergence rate for $\hat{w}_{h}$ is not reduced, but it is not optimal.  When $m=r=3$, $\bsigma$ converges with $O(h^{5/2})$ (consistent with \cref{rem:sub_parametric}), yet $\hat{w}_{h}$ performs better.  The ``improved'' error convergence for $\hat{w}_{h}$ could be due to the particular choice of exact solution.

\begin{table}[htbp]
{\footnotesize
\caption{EoC for the simply supported disk.  $N_T$ is the number of triangles in the final mesh after multiple uniform refinements.  Asterisks indicate suboptimal orders and italics indicate $m=r+1$.}\label{tab:simpsupp_disk_EOC}
\begin{center}
\begin{tabular}{cccK{.9in}K{.9in}K{.9in}K{.9in}} \toprule
$N_{T}$ & $m$ & $r$ & $| \hat{w}_{h} |_{H^1}$ & $| \hat{w}_{h} |_{H^2_h}$ & $| \hat{\bsigma}_{h} |_{L^2}$ & $h^{1/2} | \hatsignn_{h} |_{L^2(\EkDom{h}{m})}$ \\ \midrule
% num_elem & m & r & H^1 w & H^2 w & L^2 \sig & skel \sig
$2^{17}$ & 1 & 0 & \it 1.0002 & \it 0.0000 & \it 0.9997 & \it 1.0016 \\
$2^{15}$ & 1 & 1 & 1.0827\rlap{*} & 0.6840\rlap{*} & 0.4976\rlap{*} & 0.4835\rlap{*} \\
$2^{15}$ & 1 & 2 & 1.0297\rlap{*} & 0.4920\rlap{*} & 0.4926\rlap{*} & 0.4775\rlap{*} \\
$2^{15}$ & 2 & 1 & \it 1.9997 & \it 0.9996 & \it 1.9988 & \it 2.0202 \\ 
$2^{15}$ & 2 & 2 & 3.0001 & 1.9987 & 2.9974 & 2.9918 \\ 
$2^{13}$ & 2 & 3 & 3.9793 & 2.9819 & 2.5704\rlap{*} & 2.4795\rlap{*} \\
$2^{15}$ & 3 & 2 & \it 3.0001 & \it 1.9987 & \it 2.9976 & \it 2.9930 \\ 
$2^{13}$ & 3 & 3 & 3.9789 & 2.9820 & 2.5780\rlap{*} & 2.4783\rlap{*} \\
$2^{13}$ & 3 & 4 & 3.5159\rlap{*} & 2.5344\rlap{*} & 2.5008\rlap{*} & 2.4952\rlap{*} \\
$2^{13}$ & 4 & 3 & \it 3.9896 & \it 2.9916 & \it 4.0010 & \it 4.0354 \\ 
$2^{13}$ & 4 & 4 & 5.0107 & 4.0067 & 4.9846 & 4.9747 \\ 
$2^{13}$ & 5 & 4 & \it 5.0107 & \it 4.0067 & \it 4.9849 & \it 4.9772 \\
\bottomrule
\end{tabular}
\end{center}
}
\end{table}

%------------------------------------------------------------------------
\subsection{Three-Leaf Domain}\label{sec:three_leaf_ex}
%------------------------------------------------------------------------

The boundary of $\Dom$ is parameterized by
\begin{equation}\label{eqn:three_leaf_domain_bdy_param}
	x (t) = [1 + 0.4 \cos(3 t)] \cos(t), ~~
	y (t) = [1 + (0.4 + 0.22 \sin(t)) \cos(3 t)] \sin(t),
\end{equation}
for $0 \leq t \leq 2 \pi$ (see \cref{fig:three_leaf_domain}).  This domain does not have the additional symmetry of the disk.

\begin{figure}[ht]
\begin{center}

\includegraphics[width=4.0in]{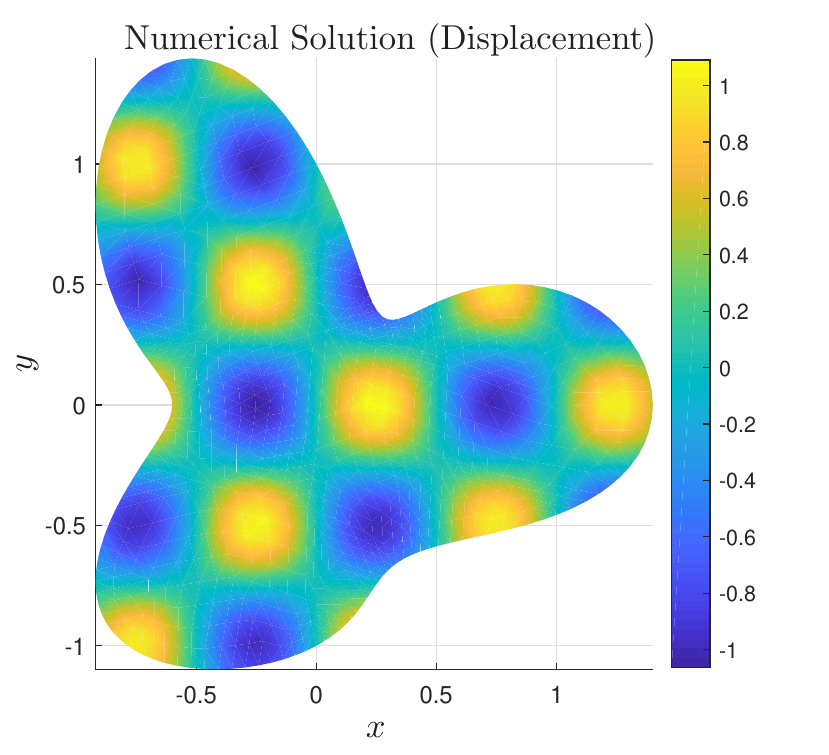}
\caption{Illustration of the three-leaf domain.  The numerical solution $\hat{w}_{h}$ approximating \cref{eqn:clamped_exact_soln_three_leaf} is shown.
}
\label{fig:three_leaf_domain}
\end{center}
\end{figure}

%------------------------------------------------------------------------
\subsubsection{Inhomogeneous Clamped Boundary Conditions}\label{sec:clamped_three_leaf_ex}
%------------------------------------------------------------------------

The exact solution is taken to be
\begin{equation}\label{eqn:clamped_exact_soln_three_leaf}
	w(x,y) = \sin( 2 \pi x) \cos( 2 \pi y),
\end{equation}
with the corresponding boundary conditions. \Cref{tab:clamped_EOC_three_leaf} shows the estimated orders of convergence (EoC).  The order of convergence is reduced, as expected, for $\hat{\bsigma}_{h}$, e.g., when $m=4$, the convergence rate \emph{goes down} (from $O(h^4)$ to $O(h^{3.5})$) when $r$ \emph{increases} from $r=3$ to $r=4$. One would expect the convergence rate to at least stay the same.  The reason for this is connected to estimating the term \cref{eqn:exact_discrete_error_eqn_pf_2} in the proof of \cref{lem:err_estim_exact_discrete}, where if $m \neq r+1$, then the geometric error is $O(h^{m-1})$.  So, in the example above, the error should go down to $O(h^3)$.  However, this geometric error is concentrated in the elements adjacent to the boundary only, so it is only $O(h^{3.5})$; see \cref{rem:sub_parametric}.  Recall that when $m=r+1$, then the Fortin property \cref{eqn:b_h_form_interp_oper_Fortin} applies, and the geometric error is $O(h^{m})$.  When $m > r+1$, then $O(h^{m-1}) = O(h^{r+1})$, which is sufficient (see the hypothesis of \cref{thm:error_estim_inhomog_BCs}).  However, the convergence order of $\hat{w}_{h}$ is better than expected, although it is reduced when $m = r-1$.

\begin{table}[htbp]
{\footnotesize
\caption{EoC for the clamped three leaf domain.  $N_T$ is the number of triangles in the final mesh after multiple uniform refinements.    Asterisks indicate suboptimal orders and italics indicate $m=r+1$.}\label{tab:clamped_EOC_three_leaf}
\begin{center}
\begin{tabular}{rccK{.85in}K{.85in}K{.85in}K{.85in}} \toprule
$N_{T}$ & $m$ & $r$ & $| \hat{w}_{h} |_{H^1}$ & $| \hat{w}_{h} |_{H^2_h}$ & $| \hat{\bsigma}_{h} |_{L^2}$ & $h^{1/2} | \hatsignn_{h} |_{L^2(\EkDom{h}{m})}$ \\ \midrule
% num_elem & m & r & H^1 w & H^2 w & L^2 \sig & skel \sig
$368640$ & 1 & 0 & \it 1.0009 & \it 0.0000 & \it 0.9993 & \it 1.0072 \\ 
$92160$ & 1 & 1 & 1.8821\rlap{*} & 0.8326\rlap{*} & 0.5016\rlap{*} & 0.4787\rlap{*} \\ 
$92160$ & 1 & 2 & 1.5191\rlap{*} & 0.5045\rlap{*} & 0.4966\rlap{*} & 0.4809\rlap{*} \\ 
$92160$ & 2 & 1 & \it 2.0009 & \it 1.0010 & \it 1.9942 & \it 2.0633 \\ 
$92160$ & 2 & 2 & 2.9808 & 1.9271\rlap{*} & 1.5048\rlap{*} & 1.5069\rlap{*} \\ 
$23040$ & 2 & 3 & 2.5621\rlap{*} & 1.5496\rlap{*} & 1.5001\rlap{*} & 1.5195\rlap{*} \\ 
$92160$ & 3 & 2 & \it 2.9996 & \it 2.0000 & \it 2.9983 & \it 2.9908 \\ 
$23040$ & 3 & 3 & 3.9900 & 2.9619 & 2.5301\rlap{*} & 2.4545\rlap{*} \\ 
$23040$ & 3 & 4 & 3.5672\rlap{*} & 2.5239\rlap{*} & 2.4872\rlap{*} & 2.4317\rlap{*} \\ 
$23040$ & 4 & 3 & \it 3.9972 & \it 2.9985 & \it 3.9933 & \it 4.1447 \\ 
$23040$ & 4 & 4 & 4.9893 & 3.9460 & 3.5022\rlap{*} & 3.5105\rlap{*} \\ 
$23040$ & 5 & 4 & \it 4.9989 & \it 4.0003 & \it 4.9959 & \it 4.9710 \\ 
\bottomrule
\end{tabular}
\end{center}
}
\end{table}

%------------------------------------------------------------------------
\subsubsection{Inhomogeneous Simply Supported Boundary Conditions}\label{sec:simsupp_three_leaf_ex}
%------------------------------------------------------------------------

The exact solution is taken to be
\begin{equation}\label{eqn:simpsupp_exact_soln_three_leaf}
	w(x,y) = \sin( 2 \pi x) \cos( 2 \pi y),
\end{equation}
with the corresponding boundary conditions. \Cref{tab:simpsupp_EOC_three_leaf} shows the estimated orders of convergence (EoC).  The convergence order is consistent with the error estimate in \cref{eqn:general_error_estim_smooth}. 
When $m=r=1$, we see $O(h^{1/2})$ for $\hat{\bsigma}_{h}$ (see \cref{rem:sub_parametric}). The convergence rate for $\hat{w}_{h}$ is not reduced, but it is not optimal.  When $m=r=3$, $\bsigma$ converges with $O(h^{5/2})$ (consistent with \cref{rem:sub_parametric}), yet $\hat{w}_{h}$ performs better.  The ``improved'' error convergence for $\hat{w}_{h}$ could be due to the particular choice of exact solution.

\begin{table}[htbp]
{\footnotesize
\caption{EoC for the simply supported three leaf domain.  $N_T$ is the number of triangles in the final mesh after multiple uniform refinements.  Asterisks indicate suboptimal orders and italics indicate $m=r+1$.}\label{tab:simpsupp_EOC_three_leaf}
\begin{center}
\begin{tabular}{cccK{.85in}K{.85in}K{.85in}K{.85in}} \toprule
$N_{T}$ & $m$ & $r$ & $| \hat{w}_{h} |_{H^1}$ & $| \hat{w}_{h} |_{H^2_h}$ & $| \hat{\bsigma}_{h} |_{L^2}$ & $h^{1/2} | \hatsignn_{h} |_{L^2(\EkDom{h}{m})}$ \\ \midrule
$368640$ & 1 & 0 & \it 1.0007 & \it 0.0000 & \it 0.9995 & \it 1.0048 \\
$92160$ & 1 & 1 & 1.1860\rlap{*} & 0.8325\rlap{*} & 0.4992\rlap{*} & 0.4680\rlap{*} \\ 
$92160$ & 1 & 2 & 0.9950\rlap{*} & 0.4968\rlap{*} & 0.4912\rlap{*} & 0.4669\rlap{*} \\ 
$92160$ & 2 & 1 & \it 2.0009 & \it 1.0010 & \it 1.9944 & \it 2.0575 \\ 
$92160$ & 2 & 2 & 2.9810 & 1.9271\rlap{*} & 1.5055\rlap{*} & 1.6780\rlap{*} \\ 
$23040$ & 2 & 3 & 2.5891\rlap{*} & 1.5470\rlap{*} & 1.5055\rlap{*} & 1.9092\rlap{*} \\ 
$92160$ & 3 & 2 & \it 2.9996 & \it 2.0000 & \it 2.9985 & \it 2.9837 \\ 
$23040$ & 3 & 3 & 3.9901 & 2.9617 & 2.5296\rlap{*} & 2.3845\rlap{*} \\ 
$23040$ & 3 & 4 & 3.5682\rlap{*} & 2.5217\rlap{*} & 2.4828\rlap{*} & 2.3819\rlap{*} \\ 
$23040$ & 4 & 3 & \it 3.9972 & \it 2.9985 & \it 3.9939 & \it 4.1086 \\ 
$23040$ & 4 & 4 & 4.9893 & 3.9457 & 3.5029\rlap{*} & 3.8619\rlap{*} \\ 
$23040$ & 5 & 4 & \it 4.9989 & \it 4.0003 & \it 4.9961 & \it 4.9544 \\ 
\bottomrule
\end{tabular}
\end{center}
}
\end{table}

%------------------------------------------------------------------------
\section{Final
Remarks}\label{sec:conclusions}
%------------------------------------------------------------------------

We have shown that the classic HHJ method can be extended to curved domains using parametric approximation of the geometry to solve the Kirchhoff plate problem on a curved domain.  Moreover, optimal convergence rates are achieved so long as the degree of geometry approximation $m$ exceeds the degree of polynomial approximation $r$ by at least $1$ (recall that the degree of the Lagrange space is $r+1$). Smaller values of $m$ generally lead to some deterioration of the convergence rates, although our estimates are not always sharp in this situation.

In particular, we have shown that solving the simply supported plate problem on a curved domain using \emph{polygonal} approximation of the domain and lowest order HHJ elements gives optimal first order convergence, the well-known Babu\v{s}ka paradox notwithstanding \cite{Babuska_SJMA1990}. Perhaps surprisingly, if the triangulation is sufficiently fine, the HHJ method computed on the fixed polygon will yield a good approximation of the exact solution of the plate problem on the smooth domain, \emph{not} a good approximation of the exact solution on the polygonal domain. One explanation of the Babu\v{s}ka paradox is that the polygonal approximating domains do not converge to the true domain in the sense of curvature, with curvature being crucial to the simply-supported boundary conditions.  However, our results show that the HHJ method does not require convergence of the curvatures.  In this sense, we might refer to the HHJ method as \emph{geometrically non-conforming}.

\bibliographystyle{siamplain}
\bibliography{Curved_HHJ}

\begin{thebibliography}{10}

\bibitem{Agmon_book1965}
{\sc S.~Agmon}, {\em Lectures on elliptic boundary value problems.}, Van
  Nostrand mathematical studies: no. 2, Van Nostrand, 1965.

\bibitem{Arnold_M2AN1985}
{\sc {Arnold, D. N.} and {Brezzi, F.}}, {\em Mixed and nonconforming finite
  element methods : implementation, postprocessing and error estimates}, ESAIM:
  M2AN, 19 (1985), pp.~7--32, \url{https://doi.org/10.1051/m2an/1985190100071},
  \url{https://doi.org/10.1051/m2an/1985190100071}.

\bibitem{Babuska_MC1980}
{\sc I.~Babu\v{s}ka, J.~Osborn, and J.~Pitk\"{a}ranta}, {\em Analysis of mixed
  methods using mesh dependent norms}, Mathematics of Computation, 35 (1980),
  pp.~1039--1062, \url{http://www.jstor.org/stable/2006374}.

\bibitem{Babuska_SJMA1990}
{\sc I.~Babu\v{s}ka and J.~Pitk\"{a}ranta}, {\em The plate paradox for hard and
  soft simple support}, SIAM Journal on Mathematical Analysis, 21 (1990),
  pp.~551--576, \url{https://doi.org/10.1137/0521030},
  \url{https://doi.org/10.1137/0521030}.

\bibitem{Blum_MMAS1980}
{\sc H.~Blum and R.~Rannacher}, {\em On the boundary value problem of the
  biharmonic operator on domains with angular corners}, Mathematical Methods in
  the Applied Sciences, 2 (1980), pp.~556--581,
  \url{https://doi.org/10.1002/mma.1670020416},
  \url{https://onlinelibrary.wiley.com/doi/abs/10.1002/ mma.1670020416}.

\bibitem{Blum_CM1990}
{\sc H.~Blum and R.~Rannacher}, {\em On mixed finite element methods in plate
  bending analysis. part 1: The first {Herrmann} scheme}, Computational
  Mechanics, 6 (1990), pp.~221--236, \url{https://doi.org/10.1007/BF00350239},
  \url{https://doi.org/10.1007/BF00350239}.

\bibitem{Boffi_book2013}
{\sc D.~Boffi, F.~Brezzi, and M.~Fortin}, {\em Mixed Finite Element Methods and
  Applications}, vol.~44 of Springer Series in Computational Mathematics,
  Springer-Verlag, New York, NY, 2013.

\bibitem{Brenner2004}
{\sc S.~C. Brenner}, {\em Discrete {S}obolev and {P}oincar\'{e} inequalities
  for piecewise polynomial functions.}, ETNA: Electronic Transactions on
  Numerical Analysis [electronic only], 18 (2004), pp.~42--48,
  \url{http://eudml.org/doc/124813}.

\bibitem{Brezzi_RAIRO1975}
{\sc F.~Brezzi and L.~D. Marini}, {\em On the numerical solution of plate
  bending problems by hybrid methods}, R.A.I.R.O. Analyse Num\'{e}rique, 9
  (1975), pp.~5--50, \url{https://doi.org/10.1051/m2an/197509R300051},
  \url{https://doi.org/10.1051/m2an/197509R300051}.

\bibitem{Brezzi_Cal1980}
{\sc F.~Brezzi, L.~D. Marini, A.~Quarteroni, and P.~A. Raviart}, {\em On an
  equilibrium finite element method for plate bending problems}, Calcolo, 17
  (1980), pp.~271--291.

\bibitem{Brezzi_RIA1976}
{\sc F.~Brezzi and P.~A. Raviart}, {\em Mixed finite element methods for 4th
  order elliptic equations}, in Topics In Numerical Analysis III: Proceedings
  of the Royal Irish Academy Conference on Numerical Analysis, J.~J.~H. Miller,
  ed., Academic Press, 1976, pp.~33--56.

\bibitem{Chernuka_IJNME1972}
{\sc M.~W. Chernuka, G.~R. Cowper, G.~M. Lindberg, and M.~D. Olson}, {\em
  Finite element analysis of plates with curved edges}, International Journal
  for Numerical Methods in Engineering, 4 (1972), pp.~49--65,
  \url{https://doi.org/10.1002/nme.1620040108},
  \url{https://onlinelibrary.wiley.com/doi/abs/10.1002/ nme.1620040108}.

\bibitem{Ciarlet_CMAME1972}
{\sc P.~Ciarlet and P.-A. Raviart}, {\em Interpolation theory over curved
  elements, with applications to finite element methods}, Computer Methods in
  Applied Mechanics and Engineering, 1 (1972), pp.~217 -- 249,
  \url{https://doi.org/10.1016/0045-7825(72)90006-0},
  \url{http://www.sciencedirect.com/science/article/pii/ 0045782572900060}.

\bibitem{Ciarlet_Book2002}
{\sc P.~G. Ciarlet}, {\em The Finite Element Method for Elliptic Problems},
  Classics in Applied Mathematics, SIAM, Philadelphia, PA, 2nd~ed., 2002.
\newblock ISBN: 978-0898715149.

\bibitem{Comodi_MC1989}
{\sc M.~I. Comodi}, {\em The {Hellan-Herrmann-Johnson} method: Some new error
  estimates and postprocessing}, Mathematics of Computation, 52 (1989),
  pp.~17--29, \url{http://www.jstor.org/stable/2008650}.

\bibitem{Krendl_ETNA2016}
{\sc W.~Krendl, K.~Rafetseder, and W.~Zulehner}, {\em A decomposition result
  for biharmonic problems and the {Hellan-Herrmann-Johnson} method}, Electronic
  Transactions on Numerical Analysis, 45 (2016), pp.~257--282.

\bibitem{Landau_Book1970}
{\sc L.~D. Landau and E.~M. Lifshitz}, {\em Theory of Elasticity}, vol.~7 of
  Course of Theoretical Physics, Addison-Wesley, 2nd~ed., 1970.

\bibitem{Lenoir_SJNA1986}
{\sc M.~Lenoir}, {\em Optimal isoparametric finite elements and error estimates
  for domains involving curved boundaries}, SIAM Journal of Numerical Analysis,
  23 (1986), pp.~562--580.

\bibitem{Mansfield_SJNA1978}
{\sc L.~Mansfield}, {\em Approximation of the boundary in the finite element
  solution of fourth order problems}, SIAM Journal on Numerical Analysis, 15
  (1978), pp.~568--579, \url{http://www.jstor.org/stable/2156585}.

\bibitem{Monk_SJNA1987}
{\sc P.~Monk}, {\em A mixed finite element method for the biharmonic equation},
  SIAM Journal on Numerical Analysis, 24 (1987), pp.~737--749,
  \url{http://www.jstor.org/stable/2157586}.

\bibitem{Rafetseder_SJNA2018}
{\sc K.~Rafetseder and W.~Zulehner}, {\em A decomposition result for
  {Kirchhoff} plate bending problems and a new discretization approach}, SIAM
  Journal on Numerical Analysis, 56 (2018), pp.~1961--1986,
  \url{https://doi.org/10.1137/17M1118427},
  \url{https://doi.org/10.1137/17M1118427}.

\bibitem{LRScott_PhD1973}
{\sc L.~R. Scott}, {\em Finite Element Techniques For Curved Boundaries}, PhD
  thesis, Massachusetts Institute of Technology, Cambridge, 1973.

\bibitem{Scott_Report1976}
{\sc L.~R. Scott}, {\em Survey of displacement methods for the plate bending
  problem}, conference: {US-Germany} symposium on finite element analysis,
  {Cambridge, MA}, Brookhaven National Lab., Upton, N.Y. (USA), Aug 1976,
  \url{https://doi.org/https://www.osti.gov/servlets/purl/7268336}.
\newblock Report Numbers: BNL-21590; CONF-760812-3.

\bibitem{Stenberg_M2AN1991}
{\sc {Stenberg, Rolf}}, {\em Postprocessing schemes for some mixed finite
  elements}, ESAIM: M2AN, 25 (1991), pp.~151--167,
  \url{https://doi.org/10.1051/m2an/1991250101511},
  \url{https://doi.org/10.1051/m2an/1991250101511}.

\bibitem{Strang_PDEproc1971}
{\sc G.~Strang and A.~E. Berger}, {\em The change in solution due to change in
  domain}, in Partial Differential Equations (Proc. Sympos. Pure Math.),
  vol.~XXIII, Amer. Math. Soc., Providence, R.I., 1971, pp.~199--205.

\bibitem{Tiihonen_MC2001}
{\sc T.~Tiihonen}, {\em Shape calculus and finite element method in smooth
  domains}, Mathematics of Computation, 70 (2001), pp.~1--15,
  \url{http://www.jstor.org/stable/2698922}.

\bibitem{Walker_SJSC2018}
{\sc S.~W. Walker}, {\em {FELICITY}: A {Matlab}/{C++} toolbox for developing
  finite element methods and simulation modeling}, SIAM Journal on Scientific
  Computing, 40 (2018), pp.~C234--C257,
  \url{https://doi.org/10.1137/17M1128745},
  \url{https://doi.org/10.1137/17M1128745},
  \url{https://arxiv.org/abs/https://doi.org/10.1137/17M1128745}.

\bibitem{Zlamal_SJNA1973}
{\sc M.~Zl\'{a}mal}, {\em Curved elements in the finite element method. {I}},
  SIAM Journal on Numerical Analysis, 10 (1973), pp.~229--240,
  \url{https://doi.org/10.1137/0710022}, \url{https://doi.org/10.1137/0710022}.

\bibitem{Zlamal_IJNME1973}
{\sc M.~Zl\'{a}mal}, {\em The finite element method in domains with curved
  boundaries}, International Journal for Numerical Methods in Engineering, 5
  (1973), pp.~367--373, \url{https://doi.org/10.1002/nme.1620050307},
  \url{https://onlinelibrary.wiley.com/doi/abs/10.1002/ nme.1620050307}.

\bibitem{Zlamal_SJNA1974}
{\sc M.~Zl\'{a}mal}, {\em Curved elements in the finite element method. {II}},
  SIAM Journal on Numerical Analysis, 11 (1974), pp.~347--362,
  \url{https://doi.org/10.1137/0711031}, \url{https://doi.org/10.1137/0711031}.

\end{thebibliography}

\end{document}